\documentclass[reqno,11pt]{amsart}

\usepackage{xcolor}
\usepackage{graphicx}
\usepackage[hidelinks]{hyperref}
\usepackage{amscd}

\usepackage{amsmath}
\usepackage{amsthm}
\usepackage{amssymb}
\usepackage{latexsym,array}
\usepackage{amsfonts}
\usepackage{shadow}
\usepackage{amsbsy}
\usepackage{amssymb}
\usepackage{dsfont}
\usepackage{mathtools}
\usepackage{enumitem}
\usepackage{doi}
\usepackage{color}
\usepackage{graphicx}
\usepackage[hidelinks]{hyperref}
\usepackage{cleveref}
\usepackage{fullpage}
\usepackage{comment}

\usepackage{dsfont}

\usepackage{cleveref}

\numberwithin{equation}{section}

\newcommand{\unq}{\underline{q}}

\newcommand{\uns}{\underline{s}}

\newtheorem{theorem}{Theorem}[section]
\newtheorem{lemma}[theorem]{Lemma}
\newtheorem{corollary}[theorem]{Corollary}
\newtheorem{proposition}[theorem]{Proposition}

\theoremstyle{definition}
\newtheorem{remark}[theorem]{{\bf Remark}}
\newtheorem{definition}[theorem]{Definition}

\usepackage{amscd}
\usepackage{tikz-cd}
\usepackage{tikz}

\newtheorem{example}[theorem]{Example}
\newtheorem{examples}[theorem]{Examples}

\crefname{enumi}{}{}
\crefname{enumii}{}{}

\title[A variation of the inverse Fueter theorem and the generalized polyanalytic Cauchy-Kovalevskaya extension of order 2]{A variation of the inverse Fueter theorem and the generalized polyanalytic Cauchy-Kovalevskaya extension of order 2}

\author[A. De Martino]{Antonino De Martino}
\address{(ADM)
	Politecnico di Milano\\Dipartimento di Matematica\\Via E. Bonardi, 9\\20133
	Milano, Italy
} \email{antonino.demartino@polimi.it}

\author[S. Pinton]{Stefano Pinton}
\address{(SP)
	Politecnico di Milano\\Dipartimento di Matematica\\Via E. Bonardi, 9\\20133
	Milano, Italy
} \email{stefano.pinton@polimi.it}

\date{}

\begin{document}
	
\maketitle

\begin{abstract}
In this paper, we first establish a generalized Cauchy--Kovalevskaya (GCK) extension for axially polyanalytic functions of order $2$. We prove that the extension can be written as a power series involving differential operators acting on two initial functions. We also study the decomposition of the GCK extension in terms of integrals over the sphere 2-sphere $\mathbb{S}$ involving plane-wave type functions.

We further establish a connection between the GCK extension for polyanalytic functions of order $2$ and the Fueter theorem. This is one of the most important results in hypercomplex analysis and can be described in two steps. In the first step, starting from holomorphic functions of one complex variable, the application of a suitable operator yields slice hyperholomorphic functions. In the second step, applying the Laplace operator in four real variables (called Fueter map) one obtains axially monogenic functions, i.e. functions in the kernel of the Fueter operator.

A suitable factorization of the Fueter map in terms of the Fueter operator and its conjugate gives rise to two intermediate classes of functions between slice hyperholomorphic functions and axially monogenic functions: axially harmonic functions and axially polyanalytic functions of order $2$.

Another goal of this paper is to study the invertibility of the factorizations of the Fueter map for harmonic and polyanalytic functions of order 2 on suitable open sets, and to derive integral representation formulas for the inverse of the factorized Fueter map. These integral representations are based on the Cauchy formula for polyanalytic functions of order $2$ and the Poisson integral formula for harmonic functions.
\end{abstract}



\tableofcontents

\section{Introduction}

One of the most interesting generalizations of complex analysis is hypercomplex analysis. This theory has been studied since the early  thirties by various mathematicians, among whom Moisil and Fueter. They considered possible notions of analytic functions over the quaternions. Since then, Fueter and his school initiated a systematic study of these notions of “regular” functions over the quaternions.
A drawback of the theory is that the identity function $f(q)=q$, as well as general polynomials in $q$, are not regular in the sense of Fueter. However, in \cite{Fueter}, Fueter introduced a powerful approach to generate regular functions in higher dimensions based on holomorphic functions of one complex variable defined on open sets in the upper half-plane. This result is now known as the Fueter theorem. It was later generalized to the Clifford setting, in odd dimensions, by M.~Sce. (see \cite{SC} and \cite{ColSabStrupSce} for an English translation) and by T.~Qian (see \cite{Q}) in the general case. In \cite{Q, Q1, TAOBOOK}, the holomorphic function was not necessarily chosen in the upper complex plane, but rather defined on a general open set.

In \cite{KQS, PQS}, the Fueter theorem was further extended to the case in which the holomorphic function is multiplied by a Fueter regular polynomial of degree $k$. In \cite{Elb}, a version of the Fueter theorem in the biaxial setting has been investigated.
\newline
\newline
Another way to obtain axially Fueter regular functions is provided by the generalized Cauchy-Kovalevskaya (GCK) extension, which characterizes such functions in terms of their restriction to the real line and is defined in terms of powers of $\underline{q}\,\partial_{q_0}$, where $q_0$ and $\underline{q}$ are defined below. A connection between the Fueter mapping theorem and the generalized Cauchy--Kovalevskaya extension has been studied in \cite{DDG, DDG1}.
\newline
\newline
The setting in which we work is that of real quaternions, which are defined as
$$ \mathbb{H}:= \{q=q_0+e_1 q_1+e_2 q_2+e_3q_3 \, | \, q_0, q_1, q_2, q_3 \in \mathbb{R}\},$$
where the imaginary units satisfy the relations
$$ e_1^2=e_2^2=e_3^2=-1,$$
$$ e_1 e_2=-e_2e_1=e_3, \quad e_2e_3=-e_3 e_2=e_1, \quad  e_3e_1=-e_1e_3=e_2.$$
A quaternion can also be written as $q=q_0+ \underline{q}$, where we denote by $q_0$ its real part and $\underline{q}=e_1 q_1+e_2q_2+e_3 q_3$ its imaginary unit. The conjugate of a quaternion $q \in \mathbb{H}$ is defined as $\bar q=q_0-\underline{q}$ and its modulus is given by
$|q|=\sqrt{q \bar{q}}=\sqrt{q_0^2+q_1^2+q_2^2+q_3^2}$.
\newline
\newline
By using the above notation, we can give a more precise formulation of the Fueter theorem. Let $D$ be an open subset of $\mathbb{C}$. Let $f(z)=u(x,y)+iv(x,y)$ be a holomorphic function on $D$. We denote this class of functions by $ \mathcal{O}(D)$. We define $\Omega_D$, the open set induced by $D$, as
$$
\Omega_D= \left\{ q=q_0+\underline{q} \in \mathbb{H} \,:\, (q_0, |\underline{q}|) \in D \right\}.
$$
The Fueter theorem can be described in two steps. In the first step, the holomorphic function $f$ is transformed into a left slice hyperholomorphic function by substituting $x$ with $q_0$, $y$ with $|\underline{q}|$, and $i$ with $\frac{\underline{q}}{|\underline{q}|}$ (see Definition~\ref{sh} for a precise definition of this class of functions). This operator is denoted by $S$, see \eqref{slice1}, and we denote the resulting class of functions by $\mathcal{SH}_L(\Omega_D)$.

In the second step, the Laplace operator in four real variables,
$
\Delta_4=\partial_{q_0}^2+\partial_{q_1}^2+\partial_{q_2}^2+\partial_{q_3}^2,
$
is applied to the left slice hyperholomorphic function. Thus, the function
$$
F(q)=\Delta_4 \left( u(q_0, |\underline{q}|)+\frac{\underline{q}}{|\underline{q}|}v(q_0, |\underline{q}|)\right)
$$
is both left and right axially Fueter regular, i.e., it belongs to the kernel of the operator
$$
D=\partial_{q_0}+\partial_{q_1} e_1+\partial_{q_2} e_2+\partial_{q_3} e_3.
$$
We denote this class of functions by $\mathcal{AM}_L(\Omega_D)$. Thus, the Fueter mapping theorem can be illustrated by the following diagram:
\begin{equation}
\label{scheme}
	\begin{CD}
	\textcolor{black}{\mathcal{O}(D)} @>S>> \textcolor{black}{\mathcal{SH}_L(\Omega_D)} @> \Delta_4  >> \textcolor{black}{\mathcal{AM}_L(\Omega_D)}.
\end{CD}
\end{equation}

Having established the framework, we are now in a position to introduce the notion of \emph{fine structures} in the quaternionic setting. These are based on two distinct factorizations of the Laplace operator $\Delta_4$, namely
\begin{equation}
\label{fact}
\Delta_4= D \overline{D}=\overline{D}D,
\end{equation}
where
$$
\overline{D}=\partial_{q_0}-\partial_{q_1} e_1-\partial_{q_2} e_2-\partial_{q_3} e_3.
$$
In \cite{CDPS, Polyf1, Polyf2}, the authors investigate how the factorizations in \eqref{fact} affect the Fueter theorem, and in particular how the diagram in \eqref{scheme} can be decomposed accordingly. The application of the operators $D$ and $\overline{D}$ to the class of slice hyperholomorphic functions yields the following factorization of \eqref{scheme}:
\begin{equation}
\label{scheme1}
\begin{CD}
	\mathcal{SH}_L(\Omega_D) @>\overline{D}>> \mathcal{APA}^L_{2, \overline{D}}(\Omega_D) \\
	@V{D}VV   @VV{D}V \\
	\mathcal{AH}^L_D(\Omega_D) @>\overline{D}>> \mathcal{AM}(\Omega_D).
\end{CD}
\end{equation}
where
\begin{equation}
	\label{spaceH}
	\mathcal{AH}_{D}^L(\Omega_D)=D (\mathcal{SH}_L(\Omega_D))= \{g \in C^\infty(\Omega_D) \, : \, g=Df, \, \, f \in \mathcal{SH}_L(\Omega_D)\},
\end{equation}
and
\begin{equation}
	\label{spaceP}
	\mathcal{APA}_{2, \overline{D}}^L(\Omega_D)= \overline{D}(\mathcal{SH}_L(\Omega_D))=\{g \in C^\infty(\Omega_D) \, : \,\overline{D}f, \, \, f \in \mathcal{SH}_L(\Omega_D)\}.
\end{equation}
By the Fueter mapping theorem, the functions in \eqref{spaceH} and \eqref{spaceP} are, respectively, axially harmonic and axially polyanalytic of order $2$ (i.e., they belong to the kernel of $D^2$).
\newline
\newline
Thus, the factorizations of the Laplace operator outline two different classes of functions which, together with their corresponding functional calculi based on the $S$-spectrum, constitute the notion of \emph{fine structures}; see \cite{CPS, AD, DPS}.
\newline
\newline
The study of polyanalytic functions in the complex setting has been widely explored in the literature; see \cite{Balk1, Balk2, V1}. These functions were first introduced by Kolossov in connection with problems in elasticity theory; see \cite{K1, M1}. They have also several applications, for instance in signal processing (in particular within the framework of Gabor frames), quantum mechanics, and various other fields; see \cite{A1, AF,CDDS} for an overview.
\newline
\newline
In \cite{DG}, the authors establish a connection between the first column of the scheme \eqref{scheme1} and a GCK extension for the class of axially harmonic functions. They also prove a decomposition of this extension in terms of plane-wave type functions.

\medskip

In this paper, we aim to address the following problem:

\medskip

\textbf{P1:} Construct a GCK extension for polyanalytic functions of order $2$, derive its decomposition in terms of plane-wave type functions, and establish a connection with the first row of the scheme \eqref{scheme1}.

\medskip

Another way to generate axially monogenic functions is via the dual Radon transform $\breve{R}$; see \cite{HOCH}. In \cite{CLSSmathAn}, it was proved that $\breve{R}$ maps the class of slice hyperholomorphic functions to the class of axially Fueter regular functions. Moreover, in \cite{DDG}, a connection between the dual Radon transform, the Fueter mapping theorem, and the GCK extension was established.
In this paper, we are able to express the GCK extension for polyanalytic functions of order $2$ in terms of plane-wave type functions. This provides insight into a possible definition of a Radon-type transform that maps the space of slice hyperholomorphic functions to the space defined in \eqref{spaceP}.

\medskip

In \cite{CSSOinverse, CSSinve}, the authors prove that the operator $\Delta_4$ applied to the class of slice hyperholomorphic functions is surjective onto the set of axially Fueter regular functions. A natural question arising in the study of factorizations of the Fueter map is whether the operators $D$ and $\overline{D}$ are surjective onto the sets of axially harmonic functions and axially polyanalytic functions of order $2$, respectively. In other words, in this paper we address the following problems:

\medskip

\textbf{P2:} Does the application of the operators $D$ and $\overline{D}$ to the class of slice hyperholomorphic functions yield the entire set of axially harmonic functions and axially polyanalytic functions of order $2$, respectively?

\medskip

\textbf{P3:} If the answer to the previous question is negative, is it possible to define suitable subspaces of axially harmonic functions and axially polyanalytic functions of order $2$ such that, together with the spaces $\mathcal{AH}_D^L(\Omega_D)$ and $\mathcal{APA}_{2, \overline{D}}^L(\Omega_D)$, they recover the full spaces of functions?

\medskip

In this paper we show that the answer to \textbf{P2} is negative; it is enough to consider the function $f(q)=\underline{q}$ (see Theorem~\ref{t2} and Theorem~\ref{t2_2}). However, we are able to provide a precise answer to \textbf{P3} by using the GCK extension for axially polyanalytic functions of order $2$ and axially harmonic functions. We also derive integral expressions for the inverse of the operators $D$ and $\overline{D}$ acting on the class of slice hyperholomorphic functions, based respectively on the Poisson formula for harmonic functions and the Cauchy formula for polyanalytic functions of order $2$.

\medskip

Finally, we study the surjectivity of another operator that produces axially polyanalytic functions of order $2$. In this case, the input function space is given by the class of slice polyanalytic functions of order $2$; see Definition \ref{polyope} for the definition of this operator.

\medskip

\emph{Outline of the paper:} The first section consists of this introduction. In Section~2, we review some basic notions from the theory of slice hyperholomorphic functions and polyanalytic functions. We also recall the Fueter mapping theorem, the GCK extension, and the fine structures. In Section~3, we establish a GCK extension for polyanalytic functions of order $2$ and express it in terms of plane-wave type functions. In Section~4, we establish a connection between the GCK extension of order $2$ and the fine structures, providing a solution to problem \textbf{P1}. In Section~5, we define a Radon-type transform that connects the class of slice hyperholomorphic functions to the class of axially polyanalytic functions of order $2$ arising from the fine structures. In Section~6 and Section~7, we study the invertibility of the operators $\overline{D}$ and $D$ applied to the class of slice hyperholomorphic functions, providing final answers to problems \textbf{P2} and \textbf{P3}. In Section~8, we study the surjectivity of another operator related to the fine structures, whose image is the class of axially polyanalytic functions of order $2$. Finally, in Section~9, we provide some final remarks.

\section{Preliminaries}

In this Section, we briefly recall the main notions and results concerning slice hyperholomorphic functions, Fueter regular functions, and polyanalytic functions that will be used throughout the paper. We also present the Fueter mapping theorem, the GCK extension for axially Fueter regular functions, and the fine structure.

\medskip

We denote by $\mathbb{S}$ the unit sphere of purely imaginary quaternions
$$
\mathbb{S} = \left\{ \underline{q} = q_1 e_1 + q_2 e_2 + q_3 e_3 \,\middle|\, q_1^2 + q_2^2 + q_3^2 = 1 \right\}.
$$
We observe that if $I \in \mathbb{S}$, then $I^2 = -1$. Hence, $I$ is an imaginary unit, and we denote by
$$
\mathbb{C}_I = \{ u + I v \,|\, u, v \in \mathbb{R} \},
$$
a copy of the complex plane isomorphic to $\mathbb{C}$.

\begin{definition}
Given an element $q \in \mathbb{H}$, we define
$$ [q]= \{w \in \mathbb{H} \, : \, w=q_0+I | \underline{q}|, \, I \in \mathbb{S}\}.$$
Let $ U \subseteq \mathbb{H}$. We say that
\begin{itemize}
\item the set $U$ is \emph{axially symmetric} if $[q] \subset U$ for any $q \in U$
\item the set $U$ is a \emph{slice domain} if $U \cap \mathbb{R} \neq \emptyset$ and if $U \cap \mathbb{C}_I$ is a domain in $\mathbb{C}_I$ for any $I \in \mathbb{S}$.
\end{itemize}
\end{definition}

\begin{definition}
\label{axial1}
	Let $U \subseteq \mathbb{H}$ be an axially symmetric domain. We set
\begin{equation}
\label{set}
	\mathcal{U}:=\{(q_0, |\underline{q}|) \in \mathbb{R}^2 \, : \, q_0+I|\underline{q}| \in U \ \text{for all } I \in \mathbb{S}\}.
\end{equation}
	We say that a function $f:U \to \mathbb{H}$ is a left (resp.\ right) axial function (or slice function) if there exist two functions $A$, $B: \mathcal{U} \to \mathbb{H}$ satisfying the even--odd conditions
	\begin{equation}
		\label{eo}
		A(q_0, |\underline{q}|)=A(q_0,-|\underline{q}|), \qquad
		B(q_0, |\underline{q}|)=-B(q_0,- |\underline{q}|),
		\qquad \forall (q_0, |\underline{q}|) \in \mathcal{U}.
	\end{equation}
	Moreover, the function $f$ can be written in the form
	\begin{equation}
		f(q)=A(q_0,|\underline{q}|)+\underline{\omega}B(q_0, |\underline{q}|),
		\quad
		\left( \text{resp. } f(q)=A(q_0, |\underline{q}|)+B(q_0, |\underline{q}|)\underline{\omega}\right),
		\quad
		\underline{\omega}=\frac{\underline{q}}{|\underline{q}|}.
	\end{equation}
\end{definition}

\begin{remark}
	In the literature, the terminology \emph{axial function} or \emph{slice function} varies depending on the context in which these functions are considered.
\end{remark}

\begin{definition}
\label{sh}
	Let $U \subseteq \mathbb{H}$ be an axially symmetric open set and let $\mathcal{U}$ be defined as in \eqref{set}.
	A function $f: U \to \mathbb{H}$ is called a left (resp.\ right) slice hyperholomorphic function if it is slice (see Definition~\ref{axial1}) and can therefore be written as
	\begin{equation}
		\label{slice}
		f(q)=\alpha(u,v)+I \beta(u,v),
		\quad
		\left(\text{resp. } f(q)=\alpha(u,v)+ \beta(u,v)I\right),
		\qquad \text{for } q=u+Iv \in U,
	\end{equation}
	where the functions $\alpha$, $\beta : \mathcal{U} \to \mathbb{H}$ satisfy the even--odd conditions \eqref{eo} together with the Cauchy--Riemann equations
	\begin{equation}
		\partial_u \alpha(u,v)=\partial_v \beta(u,v),
		\qquad
		\partial_v \alpha(u,v)=-\partial_u \beta(u,v).
	\end{equation}
	We denote the set of left (resp.\ right) slice hyperholomorphic functions by $\mathcal{SH}_L(U)$ (resp.\ $\mathcal{SH}_R(U)$).
\end{definition}

\begin{definition}
	A slice hyperholomorphic function \eqref{slice} for which $\alpha$ and $\beta$ are real-valued functions is called an intrinsic slice hyperholomorphic function. The set of such functions is denoted by $\mathcal{N}(U)$.
\end{definition}

The components $\alpha$ and $\beta$ of a slice hyperholomoprhic functions satisfy the following property, see \cite{CDP25,CDS2025}.
\begin{lemma}
\label{comp}
Let $U \subseteq \mathbb{H}$ be an axially symmetric open set. We assume that $f(q)=\alpha(u,v)+I\beta(u,v)$ is a slice hyperholomorphic function in $U$. Then in a suitable neighbourhood of $U \cap \mathbb{R}$ we can write the functions $\alpha(u,v)$ and $\beta(u,v)$ as
$$\alpha(u,v)= \sum_{j=0}^\infty \frac{(-1)^j v^{2j}}{(2j)!} \partial_u^{2j}[\alpha(u,0)],$$
$$\beta(u,v)= \sum_{j=0}^\infty \frac{(-1)^j v^{2j+1}}{(2j+1)!} \partial_u^{2j+1}[\alpha(u,0)].$$
\end{lemma}

\begin{remark}
	Let $U \subseteq \mathbb{H}$ be an axially symmetric slice domain. Suppose that $g \in \mathcal{SH}_L(U)$. Then the function $g$ admits the following power series expansion
	\begin{equation}
		\label{exp}
		g(q)= \sum_{\ell=0}^{\infty} \frac{1}{\ell!} \underline{q}^\ell \partial_{q_0}^\ell  g(q_0).
	\end{equation}
	Moreover, the series converges in a suitable ball $B(q_0,r)$ centred at $q_0 \in U \cap \mathbb{R}$ with radius $r>0$, see \cite{CSSisrael1, CSSisrael2}.
\end{remark}

The study of slice hyperholomorphic functions is nowadays very rich and has been extensively investigated; see \cite{bookCG, CGK, CSS, CSS2}.

Another prominent class of functions in hypercomplex analysis is the following.

\begin{definition}
Let $U \subseteq \mathbb{H}$ be an open set and let  $f \colon U \to \mathbb{H}$ be a function of class $\mathcal{C}^1(U)$.  The function $f$ is said to be left (resp.\ right) Fueter regular on $U$ if
$$
Df(q)= (\partial_{q_0}+\partial_{\underline{q}})f(q)= \left(\partial_{q_0}+ e_1 \partial_{q_1}
+ e_2 \partial_{q_2}+ e_3 \partial_{q_3}\right) f(q)= 0,
\qquad
\text{(resp.\ } f(q)D=0\text{)}.
$$
The differential operator $D$ is referred to as the Fueter operator.
\end{definition}

In what follows, we shall make use of the conjugate Fueter operator, defined by
$$ \overline{D}f(q)=(\partial_{q_0}-\partial_{\underline{q}})f(q)=\left(\partial_{q_0}-e_1 \partial_{q_1}-e_2 \partial_{q_2}-e_3 \partial_{q_3}\right)f(q).$$

An interesting subclass of Fueter regular functions, which will play a central role in this paper, is given by the so-called axially Fueter regular functions.

\begin{definition}[Axially Fueter regular functions]
Let $U$ be an axially symmetric open subset of $\mathbb{H}$. A function $f : U \to \mathbb{H}$ is called axially left (resp. right) Fueter regular if it is left (resp. right) Fueter regular and of axial type, see Definition~\ref{axial1}. The set of left (resp. right) axially Fueter regular functions is denoted by $\mathcal{AM}_L(U)$ (resp. $\mathcal{AM}_R(U)$).
\end{definition}

This class of functions has been extensively studied in \cite{red, green}.

\begin{examples}
\begin{enumerate}
\item The Cauchy kernel is an example of an axially Fueter regular function:
\begin{equation}
\label{kernelC}
E(q) = \frac{\bar{q}}{|q|^4}, \qquad q \neq 0.
\end{equation}
\item The Clifford-Appell polynomials, given by
\begin{equation}
\label{CliffApp}
\mathcal{Q}_n(q) = \frac{2}{(n+1)(n+2)} \sum_{j=0}^n (n-j+1)\, q^{\,n-j} \bar{q}^j,
\qquad n \geq 0,
\end{equation}
are examples of axially Fueter regular functions. They were initially introduced in \cite{CFM, CMF} and later studied in more detail in \cite{ACDDS, DKS}.
\end{enumerate}
\end{examples}

\begin{remark}
The polynomials $ \mathcal{Q}_n(q)$ are an Appell-sequence with respect to the conjugate Fueter operator, namely they satisfy
\begin{equation}
\frac{\overline{D}}{2} \mathcal{Q}_n(q)=n \mathcal{Q}_{n-1}(q).
\end{equation}
Moreover, the restriction of the Clifford-Appell polynomials to the real axis lead to
\begin{equation}
\label{rest}
\lim_{|\underline{q}| \to 0} \mathcal{Q}_n(q)=q_0^n.
\end{equation}

\end{remark}

In what follows, we recall the action of the operators $D$ and $\overline{D}$ on slice functions and show that they preserve the class of axial functions; see \cite{Dixan}.

\begin{lemma}
\label{apDbarD}
Let $f$ be a function of axial type, see Definition \ref{axial1}.
Then the application of the operator $D$ on the function $f$ can be written as
\begin{equation}
\label{axialD}
Df(q)= \left(\partial_{q_0}\alpha(q_0,v)-\partial_v \beta(q_0, v)- \frac{2}{v} \beta(q_0,v) \right)+ \underline{\omega} \left( \partial_{q_0} \beta(q_0,v)+\partial_v \alpha(q_0,v)\right),
\end{equation}
where $q=q_0+\underline{\omega}v$. The application of the operator $\overline{D}$ on the function $f$ can be written as

\begin{equation}
	\label{appbar}
	\overline{D}f(q)= \left(\partial_{q_0}\alpha(q_0,v)+\partial_v \beta(q_0, v)+ \frac{2}{v} \beta(q_0,v) \right)+ \underline{\omega} \left( \partial_{q_0} \beta(q_0,v)-\partial_v \alpha(q_0,v)\right).
\end{equation}

\end{lemma}

\subsection{Axially polyanalytic Fueter regular functions}

The class of polyanalytic functions provides another possible generalization of holomorphic functions; they are defined as the null-solutions of powers of the Cauchy-Riemann operator.

In the complex setting, this class of functions has been extensively studied; see, for instance, \cite{Balk1, Balk2, V1}. Polyanalytic functions also play an important role in time-frequency analysis, see \cite{A1, AF}, and in duality theorems, see \cite{CDDS}.

The class of polyanalytic functions has been generalized to the quaternionic setting in the context of slice analysis, see \cite{ADS, ADS2019}, and earlier in the framework of Fueter regular functions, see \cite{B1976}.

\begin{definition}
\label{polyorder2}
Let $U \subseteq \mathbb{H}$ be an open set. A function $f \colon U \to \mathbb{H}$ belonging to $\mathcal{C}^n(U)$ is said to be left (resp. right) axially polyanalytic Fueter regular of order $n$ on $U$ if it is of  left (resp, right) axial type (see Definition \ref{axial1}) and satisfies
$$
D^n f(q)=0,
\qquad
\text{(resp.\ } f(q)D^n=0\text{)},
\qquad \forall\, q \in U.
$$
We denote this class of functions by $ \mathcal{APA}_n^L(U)$ (resp. $ \mathcal{APA}_n^R(U)$).
\end{definition}

Functions that are axially polyanalytic Fueter regular of order $n$ can be expressed as finite sums involving Fueter regular functions and powers of $q_0$. The precise statement is given below.

\begin{theorem}[Polyanalytic decomposition]
\label{polydeco}
Let $U \subseteq \mathbb{H}$ be an  axially symmetric open set. A function $f \colon U \to \mathbb{H}$ is left (resp.\ right) axially polyanalytic Fueter regular of order $n$ if and only if there exist unique left (resp.\ right) axially Fueter regular functions $f_0,\ldots,f_{n-1}$ satisfying
	\begin{equation}
		\label{deco}
		f(q)=\sum_{k=0}^{n-1} q_0^k\, f_k(q).
	\end{equation}
\end{theorem}

\begin{remark}
In the literature, polyanalytic functions of order two play a very important role. These functions are connected with elasticity problems; see \cite{K1, M1}.
\end{remark}

We now aim to determine how the left (resp.\ right) axially Fueter regular functions appearing in the polyanalytic decomposition \eqref{deco} can be expressed in terms of the function $f$. For the purposes of this paper, it suffices to carry out this analysis for axially polyanalytic Fueter regular functions of order $2$.

\begin{proposition}
\label{decof}
Let $U \subseteq \mathbb{H}$ be an open set. A function $f \colon U \to \mathbb{H}$ is left (resp.\ right) axially polyanalytic Fueter regular of order $2$ if and only if it admits the representation
$$
f(q)= g_1(q)+ q_0\, g_2(q),
$$
where
\begin{equation}
	\label{g12}
	g_1(q)= f(q)- q_0\, Df(q),
	\qquad
	g_2(q)= Df(q).
\end{equation}
\end{proposition}
\begin{proof}
By Theorem~\ref{polydeco}, it suffices to show that the functions $g_1(q)$ and
$g_2(q)$ are left axially Fueter regular. The fact that the functions $g_1(q)$ and $g_2(q)$ are of axial type follows from the assumption that $f$ is of axial type and from \eqref{axialD}. Now, using the fact that $f$ is left axially polyanalytic Fueter regular of order~$2$ and applying the Leibniz rule, we obtain
$$
D g_1(q)= D f(q)- D\!\left[q_0 Df(q)\right]= Df(q)-Df(q)- q_0 D^2 f(q)= 0,
$$
and
$$
D g_2(q)= D^2 f(q)=0.
$$
If $f$ is assumed to be right axially polyanalytic Fueter regular of order $2$, analogous computations show that the functions $g_1(q)$ and $g_2(q)$ are right Fueter regular.
\end{proof}

\begin{remark}
Proposition~\ref{decof} can be extended to polyanalytic Fueter regular functions of order~$n$; however, the proof would be considerably more involved and will be addressed in a future paper.
\end{remark}

To state the next result, we recall that the polyanalytic Cauchy kernel of order $n$ is defined by means of the functions
\begin{equation}
\label{polyC}
C_j(q)= \frac{q_0^j}{j!}\, E(q), \qquad 0 \le j \le n-1,
\end{equation}
where $E(q)$ denotes the Cauchy kernel introduced in \eqref{kernelC}. The above kernel is useful for obtaining a polyanalytic Cauchy formula for polyanalytic functions of order $n$. This result was proved in \cite{B1976}.
\begin{theorem}[Polyanalytic Cauchy formula]\label{t1}
Let $f$ be a left axially polyanalytic Fueter regular of order $n$ in $U \subseteq \mathbb{H}$. Then, for every open set $\Sigma \subset U$, and for $s \in \Sigma$, we have
\begin{equation}
\label{polyCauchy}
f(s)=\int_{\partial \Sigma} \sum_{j=0}^{n-1} (-1)^j C_j(q-s) d\sigma_q D^j f(q)
\end{equation}
where $\partial \Sigma$ is a 3-dimensional compact smooth manifold in $U$, and $d \sigma_q$ is the differential form given by $d \sigma_q= \sum_{j=0}^3 (-1)^j e_j d \hat{q}_j$ where $d \hat{q}_j= dq_0 \wedge \cdots \wedge \widehat{dq_j} \wedge \cdots \wedge dq_3$.

\end{theorem}

\subsection{Fueter theorem, GCK-extension and fine structures}

In this section, we first recall how to construct slice hyperholomorphic functions starting from intrinsic holomorphic functions. Secondly, we recall how axially Fueter regular functions can be obtained from intrinsic holomorphic functions.

\begin{definition}
	An open connected subset of the complex plane is called an intrinsic complex domain if it is symmetric with respect to the real axis.
\end{definition}

\begin{definition}
	A holomorphic function $f(z)=\alpha(u,v)+i\beta(u,v)$, with $z=u+iv$, is called intrinsic if it is defined on an intrinsic complex domain $D$ and satisfies $\overline{f(z)}=f(\bar{z})$. We denote the set of intrinsic holomorphic functions on $D$ by $\mathcal{H}(D)$.
\end{definition}

A connection between intrinsic holomorphic functions and slice hyperholomoprhic functions, see Definition \ref{sh}, is given by the following operator.

\begin{definition}
Let $D \subset \mathbb{C}$ be an intrinsic complex domain. We set
\begin{equation}
\label{omega}
	\Omega_D:= \{q=q_0+\underline{q} \, : \, (q_0, | \underline{q}|) \in D\}.
\end{equation}

The slice operator is defined as follows:
\begin{equation}
	\label{slice1}
	S: \mathcal{H}(D) \otimes \mathbb{H} \to \mathcal{SH}_L(\Omega_D),
	\qquad
	\alpha(u,v)+i \beta(u,v) \mapsto \alpha(q_0, |\underline{q}|)+I\beta(q_0, |\underline{q}|),
\end{equation}
and consists in replacing the complex variable $z=u+iv$ with the quaternionic variable $q=q_0+\underline{q}$, and the complex imaginary unit $i$ with $I:= \frac{\underline{q}}{|\underline{q}|}$.

The notation $\mathcal{H}(D) \otimes \mathbb{H}$ denotes the space of intrinsic holomorphic functions on $D$ with values in $\mathbb{H}$.
\end{definition}

\begin{definition}
\label{analy}
	Let $D$ be an intrinsic complex domain. We set $\tilde{D}:=D \cap \mathbb{R}$. The space of real-valued analytic functions defined on $\tilde{D}$ that admit a unique holomorphic extension in $D$ is denoted by $\mathcal{A}(\tilde{D})$.
\end{definition}

The holomorphic extension map is defined as $C=\exp(i|\underline{q}|\,\partial_{q_0})$, and the slice regular extension map is defined as $S_1=S \circ C=\exp(\underline{q}\,\partial_{q_0})$, which is given by
\begin{equation}
	\label{regex}
	S_1: \mathcal{A}(\tilde{D}) \otimes \mathbb{H} \to \mathcal{SH}_L(\Omega_1), \qquad f_0(q_0) \mapsto \sum_{j=0}^{\infty} \frac{\underline{q}^j}{j!} \partial_{q_0}^j f_0(q_0),
\end{equation}
where the set $\Omega_1$ is given by
\begin{equation*}
	\Omega_1=  \{(q_0, \underline{q}) \in \mathbb{R}^4 \, \, |\, \, q_0 \in \tilde{D}, \, \, \, |\underline{q}|< R(q_0)\},
\end{equation*}
where $R(q_0)$ is the radius of a disc centered in $q_0$ and contained in $D$.

 Thus, by means of the following result, real analytic functions of one variable (with unique holomorphic extension) can be extended to slice hyperholomorphic functions on a suitable open set.

\begin{theorem}
\label{iso}
Let $D$ be an intrinsic complex domain and set $\tilde{D}:=D \cap \mathbb{R}$.
Recalling the definition of the set $\Omega_1$, we have the following isomorphisms:
$$
\mathcal{SH}_L(\Omega_1) \simeq \mathcal{A}(\tilde{D}) \otimes \mathbb{H} \simeq \mathcal{H}(D) \otimes \mathbb{H},
$$
where the notation $\mathcal{A}(\tilde{D}) \otimes \mathbb{H}$ denotes the space of real-analytic functions on $\tilde{D}$ with values in $\mathbb{H}$.
\end{theorem}

\begin{remark}
A result analogous to Theorem~\ref{iso} can be established for right slice hyperholomorphic functions; for this, it is sufficient to define a slice operator for right slice hyperholomorphic functions.
\end{remark}

One of the most fundamental results in hypercomplex analysis is the Fueter mapping theorem, see \cite{Fueter}, which is also a primary tool for transforming analytic functions of one real variable into axially Fueter regular functions.

\begin{theorem}[Fueter mapping theorem]
\label{FMT}
Let $f_0(z)=\alpha(u,v)+i\beta(u,v)$ be a holomorphic function defined on a domain $D$ in $ \mathbb{C}$, where $\alpha(u,v)$ and $\beta(u,v)$ satisfy \eqref{eo}, and let $\Omega_D$ be as in \eqref{omega}. Then
$$
\Delta_4 \circ S[f_0] = \Delta_4 f(q_0+\underline{q}) = \Delta_4 \left( \alpha(q_0, |\underline{q}|) + \frac{\underline{q}}{|\underline{q}|} \beta(q_0, |\underline{q}|) \right)
$$
is left axially Fueter regular, where $\Delta_4 = \partial_{q_0}^2 + \partial_{q_1}^2 + \partial_{q_2}^2 + \partial_{q_3}^2$ is the Laplace operator in four real variables $q_\ell$, $\ell=0,1,2,3$, and $S$ is the slice operator, see \eqref{slice1}. The operator $\Delta_4$ is called the Fueter map.
\end{theorem}

By using Theorem \ref{iso} we can visualize the Fueter-mapping theorem in the following way:

\begin{equation}
	\label{SC}
	\begin{CD}
		\textcolor{black}{\mathcal{A}(\tilde{D}) \otimes \mathbb{H}} @>S >> \textcolor{black}{\mathcal{SH}_L(\Omega_1)} @> \Delta_4  >> \textcolor{black}{\mathcal{AM}_L(\Omega_1)}.
	\end{CD}
\end{equation}


\begin{remark}
	In the Fueter mapping theorem we consider only the case of left axially Fueter regular functions, but it is also possible, with minimal modifications, to treat the case of right axially Fueter regular functions. In the sequel of the paper, whenever it is not necessary to distinguish between the two cases, we will restrict ourselves to the left case.
\end{remark}

In \cite{CDPS}, the authors factorized the Fueter mapping to derive a product formula for a functional calculus arising from the Fueter mapping theorem; see \cite{CSSO}. These factorizations lead to the following notion.

\begin{definition}[Fine structures]
We call quaternionic fine structures the classes of functions and the associated functional calculi arising from the factorization of the Fueter map (namely $\Delta_4$). In particular, the fine structures are those obtained from the factorization of the Fueter map in terms of the operators $D$ and $\overline{D}$.
\end{definition}

According to the previous definition, in the quaternionic framework there are only two possible ways to factorize the Fueter mapping $\Delta_4$, namely
$$
\Delta_4=\overline{D}D, \qquad \text{and} \qquad \Delta_4=D\overline{D}.
$$
These two factorizations give rise to different fine structures. In \cite{CDPS}, the authors studied the factorization $\Delta_4=\overline{D}D$, which yields

\begin{equation}
	\begin{CD}
		\textcolor{black}{\mathcal{A}(\tilde{D}) \otimes \mathbb{H}} @>S>> \textcolor{black}{\mathcal{SH}_L(\Omega_1)} @> D >> \textcolor{black}{\mathcal{AH}_{D}^L(\Omega_1)}@> \overline{D}  >> \textcolor{black}{\mathcal{AM}_L(\Omega_1)},
	\end{CD}
\end{equation}
where $\mathcal{AH}_{D}^L(\Omega_1)$ denotes the range of the operator $D$ acting on the class of left slice hyperholomorphic functions, and is defined by
\begin{equation}
	\label{AAH}
\mathcal{AH}_{D}^L(\Omega_1)= \{g \in C^\infty(\Omega_1) \, : \, g=Df, \, f \in \mathcal{SH}_L(\Omega_1)\}.
\end{equation}
A function $g \in \mathcal{AH}_{D}^L(\Omega_1)$ is harmonic by virtue of the Fueter mapping theorem. Indeed, we have
$
\Delta_4 g = \Delta_4 D f = 0.
$
In \cite{Polyf1, Polyf2}, the authors instead consider the factorization $\Delta_4 = D\overline{D}$. In this case we obtain

\begin{equation}
	\label{SC1}
	\begin{CD}
		\textcolor{black}{\mathcal{A}(\tilde{D}) \otimes \mathbb{H}} @>S>> \textcolor{black}{\mathcal{SH}_L(\Omega_1)} @> \overline{D}  >> \textcolor{black}{\mathcal{APA}_{2, \overline{D}}^L(\Omega_1)}@>D  >> \textcolor{black}{\mathcal{AM}_L(\Omega_1)},
	\end{CD}
\end{equation}
where $\mathcal{APA}_{2,\overline{D}}^L(\Omega_1)$ is the space obtained as the image of the operator $\overline{D}$ acting on left slice hyperholomorphic functions, namely
\begin{equation}
\label{spacepoly}
\mathcal{APA}_{2, \overline{D}}^L(\Omega_1)= \{h \in C^\infty(U) \, | \, h=\overline{D}f, \, f \in \mathcal{SH}_L(\Omega_1)\}.
\end{equation}
A function $h \in \mathcal{APA}_{2, \overline{D}}^L(\Omega_1)$ is polyanalytic of order $2$, see Definition \ref{polyorder2}, by virtue of the Fueter mapping theorem. Indeed,
$
D^2 h = D^2 \overline{D}f = D \Delta_4 f = 0.
$

\begin{remark}
The Fueter mapping theorem was extended to the more general setting of Clifford algebras in the case of odd dimension in 1957 by M. Sce; see \cite{SC}. In this setting, the Laplace operator $\Delta_4$ is replaced by $\Delta_{n+1}^{\frac{n-1}{2}}$, where $\Delta_{n+1}$ is the Laplace operator in $n+1$ dimensions and $n$ is odd; see also \cite{ColSabStrupSce} for an English translation with commentaries.
In 1997, T. Qian showed that the Fueter-Sce theorem can also be established in even dimensions, in this case the operator $\Delta_{n+1}^{\frac{n-1}{2}}$ becomes a fractional operator; see \cite{Q, Q1}.
\end{remark}

\begin{remark}
In \cite{CDP25, CDP2026}, the notion of fine structures was extended to the Clifford setting. In this case, the factorizations are more involved.
\end{remark}

Another possible way to get an axially Fueter regular function starting from an analytic function of one real variable is by the so called generalzied Cauchy-Kovalevskaya (GCK) exstension (see \cite[Thm. 5.1.1]{green}).

\begin{theorem}[GCK-extension]
\label{gck}
Let $D$ be a convex intrinsic complex domain. We set $\tilde{D}:=D \cap \mathbb{R}$. We consider an analytic function $A_0(q_0) \in \mathcal{A}(\tilde{D}) \otimes \mathbb{H}$. Then there exists a unique sequence $ \{A_j(q_0)\}_{j=1}^\infty \subset \mathcal{A}(\tilde{D}) \otimes \mathbb{H}$ such that the series
\begin{equation}
\label{series1}
f(q)= \sum_{j=0}^\infty \underline{q}^j A_j(q_0),
\end{equation}
is convergent in an axially symmetric 4-dimensional neighbourhood $\Omega_1 \subset \mathbb{H}$ of $\tilde{D}$ and its sum is a Fueter regular function, i.e. $Df(q)=0$ in $\Omega_1$. Furthermore, the sum $f$ is  given by the expression
\begin{equation}
	\label{seriesexp}
	f(q)= \frac{\sqrt{\pi}}{2}\left( \sum_{j=0}^\infty \frac{(-1)^j |\underline{q}|^{2j} \partial_{q_0}^{2j} A_0(q_0)}{2^{2j}j! \Gamma \left(\frac{3}{2}+j\right)}+ \frac{\underline{q} }{2} \sum_{j=0}^\infty \frac{(-1)^j |\underline{q}|^{2j} \partial_{q_0}^{2j+1}A_0(q_0)}{2^{2j}j! \Gamma \left(\frac{5}{2}+j\right)} \right).
\end{equation}
The function $A_0$ is determined by
$$ A_0(q_0)=\lim_{|\underline{q}| \to 0} f(q).$$
 The function \eqref{seriesexp} is known as the GCK-extension of the function $A_0$, and it is denoted by $GCK[A_0](q)$. This extension operator defines an isomorphism between right modules:

$$ GCK \, : \, \mathcal{A}(\tilde{D}) \otimes \mathbb{H} \to \mathcal{AM}_L(\Omega_1),$$
where the set $\Omega_1$ is given by
\begin{equation}
\label{Stefano1}
\Omega_1=  \{(q_0, \underline{q}) \in \mathbb{R}^4 \, \, |\, \, q_0 \in \tilde{D}, \, \, \, |\underline{q}|< R(q_0)\},
\end{equation}
where $R(q_0)$ is the radius of a disc centred in $q_0$ and contained in $D$.
\end{theorem}

\begin{remark}
Several generalizations of the GCK extension have been studied in the literature; see for example \cite{DS1, GU}.
\end{remark}

\begin{remark}
	The expression \eqref{seriesexp} usually appears in the literature in terms of Bessel functions of the first kind. However, for our purposes, the form \eqref{seriesexp} is more convenient.
\end{remark}

In \cite{DS1}, the authors showed that the GCK-extension of axially monogenic functions admits a representation in terms of integrals over $\mathbb{S}^{2}$ involving plane wave-type functions, i.e. functions depending on the Euclidean inner product $\langle \underline{\omega}, \underline{q} \rangle$ in $\mathbb{R}^3$, where the $1$-vector $\underline{\omega}$ is independent of $\underline{q}$.

\begin{theorem}
\label{GCKplane}
Under the same assumptions as in Theorem~\ref{gck}, and writing the function $f$ as in \eqref{series1}, we have
$$ GCK[A_0](q)= \frac{1}{2 \pi} \int_{\mathbb{S}} \hbox{exp}(\langle \underline{\omega}, \underline{q} \rangle \underline{\omega} \partial_{q_0}) A_0(q_0)dS_{\underline{\omega}},$$
where $dS_{\underline{\omega}}$ is the area element of the sphere $ \mathbb{S}$.
\end{theorem}

\begin{remark}
In \cite{DDG}, using \eqref{rest} and the fact that the Clifford-Appell polynomials (see \eqref{CliffApp}) are axially Fueter regular, it is shown that
\begin{equation}
	\label{res}
	\mathcal{Q}_{n}(q)=GCK[q_0^n](q).
\end{equation}
\end{remark}

\section{On Generalized Cauchy-Kovalevskaya extension for polyanalytic functions of order 2}

In this section, our goal is to establish a generalized Cauchy-Kovalevskaya (GCK) extension for polyanalytic functions of order $2$. The idea is to examine the conditions under which a set of quaternionic-valued functions $\{A_j(q_0)\}_{j \in \mathbb{N}_0}$, defined on $\tilde{D}=D \cap \mathbb{R}$, where $D$ is an intrinsic complex domain, allows the series
$
\sum_{j=0}^{\infty} \underline{q}^j A_j(q_0)
$
to converge in a neighbourhood $\Omega \subset \mathbb{H}$ of $\tilde{D}$, in such a way that the resulting function is axially polyanalytic of order $2$.

To obtain GCK for polyanalytic functions of order $2$, we will make use of the following preliminary results; see \cite{red, green}.

\begin{lemma}
	Let $ \ell  \in \mathbb{N}$, then for $q \in \mathbb{H}$ we have
	\begin{equation}
		\label{f1}
		\partial_{\underline{q}}(\underline{q}^\ell)= c(\ell) \underline{q}^{\ell-1}, \qquad c(\ell)=\begin{cases}
			- \ell, \quad \ell \, \, \hbox{even}\\
			-(\ell+2), \quad \ell \, \, \hbox{odd},
		\end{cases}
	\end{equation}
where $\partial_{\underline{q}}= e_1 \partial_{q_1}+e_2 \partial_{q_2}+e_2 \partial_{q_3}$.
\end{lemma}

An easy conseguence of the above result is given by the following

\begin{lemma}
	Let $\ell \geq 2$,  then we have
	\begin{equation}
		\label{f2}
		\partial_{\underline{q}}^2(\underline{q}^\ell)=k(\ell) \underline{q}^{\ell-2}, \qquad k(\ell)=\begin{cases}
			\ell (\ell+1) , \quad \ell \, \, \hbox{even}\\
			(\ell-1)(\ell+2)	, \quad \ell \, \, \hbox{odd}.
		\end{cases}
	\end{equation}
\end{lemma}

\begin{proof}

We start by considering the case $\ell$ even, i.e. $\ell=2p$ with $p \in \mathbb{N}$, by the facts that $\underline{q}^2=-| \underline{q}|^2$ and by \eqref{f1}, we have
	$$ \partial_{\underline{q}}^2(\underline{q}^\ell)=-2p\partial_{\underline{q}} \left( \underline{q}^{2p-1}\right)=2p(2p+1)\underline{q}^{2p-2}=\ell(\ell+1) \underline{q}^{\ell-2}.$$
	Similarly, for the case $\ell$ odd, i.e. $\ell=2p+1$ with $p \in \mathbb{N}$, we have
	$$ \partial_{\underline{q}}^2(\underline{q}^\ell)=-(2p+3) \partial_{\underline{q}}(\underline{q}^{2p})=2p(2p+3) \underline{q}^{2p-1}=(\ell-1)(\ell+2) \underline{q}^{\ell-2}.$$
\end{proof}

\begin{theorem}
\label{pgck}
Let $D$ be a convex intrinsic complex domain. We set $\tilde{D}:=D \cap \mathbb{R}$. We consider two analytic functions $A_0(q_0)$, $A_1(q_0) \in \mathcal{A}(\tilde{D}) \otimes \mathbb{H}$. Then there exist a unique sequence of functions $ \{A_j\}_{j \in \mathbb{N}_0} \subset \mathcal{A}(\tilde{D}) \otimes \mathbb{H}$ such that the series
\begin{equation}
\label{fun1}
f(q)= \sum_{j=0}^{\infty} \underline{q}^jA_j(q_0)
\end{equation}
converges in axially symmetric $4$-dimensional neighbourhood $\Omega_1 \subset \mathbb{H}$ of $\tilde{D}$ and its sum is polyanalytic of order 2 (i.e. $D^2 f(q)=0$) in $\Omega_1$. Furthermore, the sum of $f$ is given by the axial expression
\begin{equation}
\label{funct}
f(q)= \alpha(q)+\underline{\omega} \beta(q), \quad \underline{\omega}=\frac{\underline{q}}{|\underline{q}|}
\end{equation}
where
\begin{eqnarray}
	\nonumber
	\alpha(q)&:=& \frac{\sqrt{\pi}}{2}\left(-\frac{3}{2}  \sum_{j=0}^{\infty} \frac{(-1)^j |\underline{q}|^{2j+2} \partial_{q_0}^{2j+1} A_1(q_0)}{2^{2 j} j! \Gamma \left(j+\frac{5}{2}\right)} + \frac{1}{2}  \sum_{j=0}^{\infty} \frac{(-1)^j |\underline{q}|^{2j+2} \partial_{q_0}^{2j+2}A_0(q_0)}{2^{2j} j! \Gamma \left(j+\frac{5}{2}\right)} \right.\\
	\label{real}
	&&\left.+ \sum_{j=0}^\infty \frac{(-1)^j |\underline{q}|^{2j} \partial_{q_0}^{2j} A_0(q_0)}{2^{2j}j! \Gamma \left(j+\frac{3}{2}\right)} \right),
\end{eqnarray}
and
\begin{eqnarray}
	\nonumber
	\beta(q)&:=& -  \frac{\sqrt{\pi}}{4}\left( \frac{3}{2} \sum_{j=0}^\infty \frac{(-1)^j |\underline{q}|^{2 j+3} \partial_{q_0}^{2 j+2}A_1(q_0)}{2^{2j}j! \Gamma \left(j+\frac{7}{2}\right)}- 3 \sum_{j=0}^\infty \frac{(-1)^j |\underline{q}|^{2j+1} \partial_{q_0}^{2j}A_1(q_0)}{2^{2j}j! \Gamma \left(j+\frac{5}{2}\right)} \right.\\
	\label{imm}
	&&\left. - \frac{1}{2} \sum_{j=0}^\infty \frac{(-1)^j |\underline{q}|^{2j+3} \partial_{q_0}^{2j+3} A_0(q_0)}{2^{2j} j! \Gamma \left(j+\frac{7}{2}\right)} \right).
\end{eqnarray}
The functions $A_0$ and $A_1$ are determined by the relations
\begin{equation}
\label{infun}
A_0(q_0)= \lim_{|\underline{q}| \to 0}f(q), \qquad
A_1(q_0)= -\frac{1}{3} \lim_{|\underline{q}| \to 0}\partial_{\underline{q}} f(q).
\end{equation}
The function in \eqref{funct} is called the GCK-extension of polyanalytic functions of order 2 of the couple $(A_0, A_1)$, and it is denoted by $PGCK[A_0,A_1](q)$. This extension operator defines an isomorphism between right modules:
$$ PGCK: (\mathcal{A}(\tilde{D}) \otimes \mathbb{H})^2 \to \mathcal{APA}_2^L(\Omega_1),$$
where the set $\Omega_1$ is defined as in \eqref{Stefano1}.
\end{theorem}

\begin{proof}
We first apply the operator $D^2=(\partial_{q_0}+\partial_{\underline{q}})^2=\partial_{q_0}^2+2 \partial_{q_0} \partial_{\underline{q}} +\partial_{\underline{q}}^2$ to both side of \eqref{fun1}, by \eqref{f1} and \eqref{f2}, and the hypothesis that the function $f$ is polyanalytic of order 2 we get
\begin{eqnarray*}
0&=& D^2 f(q)\\
&=& \sum_{j=0}^\infty \underline{q}^j \partial_{q_0}^2 A_j(q_0)+2 \sum_{j=1}^\infty \partial_{\underline{q}} (\underline{q}^j) \partial_{q_0} (A_j(q_0))+ \sum_{j=2}^\infty \partial_{\underline{q}}^2 (\underline{q}^j) A_j(q_0)\\
&=& \sum_{j=0}^\infty \underline{q}^j \partial_{q_0}^2 A_j(q_0)+2 \sum_{j=0}^\infty \partial_{\underline{q}} (\underline{q}^{j+1}) \partial_{q_0} A_{j+1}(q_0)+ \sum_{j=0}^\infty \partial_{\underline{q}}^2 (\underline{q}^{j+2}) A_{j+2}(q_0)\\
&=&  \sum_{j=0}^\infty \underline{q}^j \left[\partial_{q_0}^2 A_j(q_0)+2c(j+1)\partial_{q_0}A_{j+1}(q_0)+k(j+2)A_{j+2}(q_0)\right],
\end{eqnarray*}
where the constants $c(j+1)$ and $k(j+2)$ are defined in \eqref{f1} and \eqref{f2}, respectively. Thus we get the following recursion formula
$$ \partial_{q_0}^2 A_j(q_0)+2c(j+1)\partial_{q_0}A_{j+1}(q_0)+k(j+2)A_{j+2}(q_0)=0, \quad j \in \mathbb{N}_0$$
From which we get
\begin{equation}
\label{rec}
A_{j+2}(q_0)=-2 \frac{c(j+1)}{k(j+2)} \partial_{q_0} A_{j+1}(q_0)-\frac{\partial_{q_0}^2 A_j(q_0)}{k(j+2)}.
\end{equation}
From this recurrence relation, see Proposition \ref{recc} in Appendix, we get
\begin{equation}
\label{rec1}
		A_{j+2}(q_0)=
\begin{cases}
	\vspace*{3mm}
\frac{3 \partial_{q_0}^{j+1} A_1(q_0)}{j!! (j+3)!!}-\frac{(j+1)\partial_{q_0}^{j+2}A_0(q_0)}{(j+2)!!(j+3)!!}, \qquad j  \, \, \hbox{is even}\\
\frac{3 (j+2) \partial_{q_0}^{j+1} A_1(q_0)}{(j+1)!! (j+4)!!}- \frac{\partial_{q_0}^{j+2}A_0(q_0)}{(j-1)!! (j+4)!!}, \qquad j \, \, \hbox{is odd}
\end{cases}
\end{equation}

Now, we plug \eqref{rec1} in \eqref{fun1} and we get

\begin{eqnarray*}
	\nonumber
f(q)&=& \sum_{j=1}^\infty \underline{q}^{2j}A_{2j}(q_0)+ \sum_{j=1} \underline{q}^{2j+1} A_{2j+1}(q_0)+A_0(q_0)+ \underline{q}A_1(q_0)\\
\nonumber
&=& 3\sum_{j=1}^\infty \frac{(-1)^j |\underline{q}|^{2j} \partial_{q_0}^{2j-1}A_1(q_0)}{(2j-2)!! (2j+1)!!}- \sum_{j=1}^\infty \frac{(2j-1)(-1)^j | \underline{q}|^{2j} \partial_{q_0}^{2j}A_0(q_0)}{(2j)!! (2j+1)!!}\\
\nonumber
&&+3 \underline{q} \sum_{j=1}^\infty \frac{(2j+1) (-1)^j | \underline{q}|^{2j} \partial_{q_0}^{2j}A_1(q_0)}{(2j)!!(2j+3)!!}- \underline{q} \sum_{j=1}^\infty \frac{(-1)^j | \underline{q}|^{2j} \partial_{q_0}^{2j+1} A_0(q_0)}{(2j-2)!! (2j+3)!!}\\
\nonumber
&& +A_0(q_0)+ \underline{q}A_1(q_0)\\
\nonumber
&=&  3\sum_{j=1}^\infty \frac{(-1)^j |\underline{q}|^{2j} \partial_{q_0}^{2j-1}A_1(q_0)}{(2j-2)!! (2j+1)!!}- \sum_{j=0}^\infty \frac{(-1)^j | \underline{q}|^{2j}(2j-1) \partial_{q_0}^{2j}A_0(q_0)}{(2j)!! (2j+1)!!}\\
\label{starS}
&&+3 \underline{q} \sum_{j=0}^\infty \frac{(2j+1) (-1)^j | \underline{q}|^{2j} \partial_{q_0}^{2j}A_1(q_0)}{(2j)!!(2j+3)!!}- \underline{q} \sum_{j=1}^\infty \frac{(-1)^j | \underline{q}|^{2j} \partial_{q_0}^{2j+1} A_0(q_0)}{(2j-2)!! (2j+3)!!}.
\end{eqnarray*}
By using the following identities for the double factorial, $(2n-1)!!=\frac{2^n}{\sqrt{\pi}}\,\Gamma\!\left(n+\tfrac{1}{2}\right)$, $(2n)!!=2^n n!,
\quad n\in\mathbb{N},$
and by changing indices, we obtain \eqref{funct}.
\\Now we proceed to examine the convergence of the series appearing in \eqref{real} and \eqref{imm}.It suffices to establish the convergence of the series
\begin{equation}
	\label{series}
	\sum_{j=0}^{\infty}
	\frac{(-1)^j |\underline{q}|^{2j+2} \partial_{q_0}^{2j+1} A_1(q_0)}
	{2^{2j} j!\, \Gamma\!\left(j+\frac{5}{2}\right)}.
\end{equation}
The convergence of the remaining series involved in \eqref{real} and \eqref{imm} follows by analogous arguments. Let $q_0$ be an arbitrary point in $\widetilde{D}$. For any disc $B(q_0,R(q_0))$ such that $B(q_0,R(q_0))\subset D$, the Cauchy estimates imply (since $A_1(q_0)$ admits a unique holomorphic expansion; see Definition~\ref{analy}) that there exists a constant $C_{A_1}$ satisfying
$$
\bigl|\partial_{q_0}^{2j+1}A_1(q_0)\bigr|
\leq
\frac{(2j+1)!\,C_{A_1}}{R(q_0)^{2j+1}}.
$$
Therefore, the series in \eqref{series} can be estimated by
\begin{equation}
	\label{series2}
	\sum_{j=0}^{\infty}
	\frac{|\underline{q}|^{2j+2}(2j+1)!\,C_{A_1}}
	{R(q_0)^{2j+1}\,2^{2j}j!\Gamma\!\left(j+\frac{5}{2}\right)}.
\end{equation}
By applying the duplication formula for the Gamma function, we obtain
$$
2^{2j}j!\Gamma\!\left(j+\frac{5}{2}\right)= \frac{(2j+3)! \sqrt{\pi}}{8(j+1)}.
$$
It follows that the series in \eqref{series2} converges normally on the set $\Omega_1$ defined in \eqref{Stefano1}.
\end{proof}

We now present an example of a polyanalytic function of order~$2$, namely the polyanalytic Cauchy kernel of order~$2$, which can be obtained via the GCK extension for polyanalytic functions of order~$2$.
\begin{proposition}
Let $q \in \mathbb{H} \setminus \{0\}$. Then we have
\begin{equation}
\label{polyC1}
 C_1(q)= PGCK[q_0^{-2},-q_0^{-3}](q),
\end{equation}
where $C_1(q)$ denotes the polyanalytic Cauchy kernel of order $2$, defined in \eqref{polyC}.
\end{proposition}
\begin{proof}
The function $C_1(q)=q_0E(q)$ is polyanalytic of order~$2$. Therefore, by Theorem~\ref{pgck}, in order to determine its GCK extension, it suffices to identify the functions $A_0(q_0)$ and $A_1(q_0)$. We begin by computing $A_0(q_0)$. From \eqref{infun} we obtain
$$
A_0(q_0)= \lim_{|\underline{q}| \to 0} q_0 E(q)= q_0^{-2}.
$$

To determine $A_1(q_0)$, we make use of \eqref{infun} and obtain
\begin{equation}
	\label{interpgk}
	A_1(q_0)=-\frac{1}{3} \lim_{|\underline{q}| \to 0} q_0 \partial_{\underline{q}} \left(\frac{\bar{q}}{|q|^4}\right).
\end{equation}
Since the function $E(q)$ is Fueter regular (see Example~\ref{kernelC}), we have
$
(\partial_{q_0}+\partial_{\underline{q}})E(q)=0,
$
which implies $\partial_{\underline{q}}E(q)=-\partial_{q_0}E(q)$. Hence, \eqref{interpgk} can be rewritten as
\begin{equation} \label{interpgk2} A_1(q_0)= \frac{q_0}{3} \lim_{|\underline{q}| \to 0} \partial_{q_0} \left(\frac{\bar{q}}{|q|^4}\right)=\frac{q_0}{3} \lim_{|\underline{q}| \to 0}  \left(|q|^{-4}-4\bar{q}q_0|q|^{-6}\right)=-q_0^{-3}.
\end{equation}

Consequently, by Theorem~\ref{pgck} together with \eqref{interpgk} and \eqref{interpgk2}, we obtain \eqref{polyC1}.

\end{proof}

The GCK extension for polyanalytic functions of order 2 can be expressed in terms of the GCK extension for Fueter regular functions, see Theorem \ref{gck}, as shown in the following result.

\begin{proposition}
\label{decop}
	Let $f(q) = \sum_{j=0}^{\infty} \underline{q}^{\,j} A_j(q_0)$ be a function satisfying the same assumptions as in Theorem~\ref{pgck}. Then the GCK extension for polyanalytic functions of order~2 (see \eqref{funct}) can be expressed as
\begin{equation}
\label{Pgck}
f(q) = \mathrm{GCK}[g_1(q_0)](q) + q_0\, \mathrm{GCK}[g_2(q_0)](q),
\end{equation}
where
$$
g_1(q) = f(q) - q_0\, Df(q), \qquad g_2(q) = Df(q),
$$
where $\mathrm{GCK}$ denotes the GCK extension for axially Fueter regular functions, as defined in \eqref{seriesexp}.
\end{proposition}
\begin{proof}
By hypothesis we know that the function $f$ can be written as in \eqref{fun1}, thus by \eqref{f1} we can write the function $g_1(q)$ as
\begin{eqnarray*}
g_1(q)&=& \sum_{j=0}^\infty \underline{q}^j A_j(q_0)- q_0 D \left(\sum_{j=0}^\infty  \underline{q}^j A_j(q_0)\right)\\
&=&\sum_{j=0}^\infty \underline{q}^j A_j(q_0)- q_0 \left( \sum_{j=0}^\infty \underline{q}^j \partial_{q_0} A_j(q_0)+ \sum_{j=1}^\infty c(j) \underline{q}^{j-1} A_j(q_0)\right),
\end{eqnarray*}
Similarly, for the function $g_2(q)$ we have
$$ g_2(q)=\sum_{j=0}^\infty \underline{q}^j \partial_{q_0} A_j(q_0)+ \sum_{j=1}^\infty c(j) \underline{q}^{j-1} A_j(q_0),$$
where where the constant $c(j)$ has been defined in \eqref{f1}.
Thus the restrictions of the functions $g_1(q)$ and $g_2(q)$ to the real line are given by
\begin{equation}
\label{res1}
g_1(q_0)= \lim_{|\underline{q}| \to 0}g_1(q) =A_0(q_0)-q_0 \partial_{q_0} A_0(q_0)+3q_0 A_1(q_0),
\end{equation}
\begin{equation}
\label{res2}
g_2(q_0)=\lim_{|\underline{q}| \to 0} g_2(q)=\partial_{q_0} A_0(q_0)-3A_1(q_0).
\end{equation}
By Proposition \ref{decof} we know that the functions $g_1(q)$ and $g_2(q)$ are axially Fueter regular, so we can compute their GCK extensions for Fueter regular functions. We start from the scalar part of the GCK extension of the functions $g_1(q)$ and $g_2(q)$, that we denote by $GCK_0$. By using the Leibniz rule, for $j \in \mathbb{N}$, we obtain

\begin{eqnarray}
\nonumber
(|\underline{q}| \partial_{q_0})^{2 j} g_1(q_0)&=& (|\underline{q}| \partial_{q_0})^{2 j}A_0(q_0)-q_0 | \underline{q}|^{2 j} \partial_{q_0}^{2 j +1} A_0(q_0)-2 j | \underline{q}|^{2 j} \partial_{q_0}^{2 j} A_0(q_0)\\
\label{rev1}
&&+3 q_0| \underline{q}|^{2 j} \partial_{q_0}^{2 j} A_1(q_0)+6 j | \underline{q}|^{2 j} \partial_{q_0}^{2 j -1}A_1(q_0) .
\end{eqnarray}
By using the scalar part of \eqref{seriesexp} and performing a change of index from $j-1$ to $j$ in the third and last summations after the last equality, we obtain
\begin{eqnarray}
	\nonumber
GCK_0[g_1(q_0)](q)&=& \frac{\sqrt{\pi}}{2} \sum_{j=0}^\infty \frac{(-1)^j (|\underline{q}| \partial_{q_0})^{2j}g_1(q_0)}{2^{2j}j! \Gamma \left(\frac{3}{2}+j\right)}\\
\nonumber
&=& \frac{\sqrt{\pi}}{2} \sum_{j=0}^\infty \frac{(-1)^j (|\underline{q}| \partial_{q_0})^{2j}A_0(q_0)}{2^{2j}j! \Gamma \left(\frac{3}{2}+j\right)}-\frac{q_0\sqrt{\pi}}{2} \sum_{j=0}^\infty \frac{(-1)^j |\underline{q}|^{2j} \partial_{q_0}^{2j+1}A_0(q_0)}{2^{2j}j! \Gamma \left(\frac{3}{2}+j\right)}\\
\nonumber
&&+ \frac{\sqrt{\pi}}{4} \sum_{j=0}^\infty \frac{(-1)^j (|\underline{q}| \partial_{q_0})^{2j+2} A_0(q_0)}{2^{2j}j! \Gamma \left(\frac{5}{2}+j\right)}+\frac{3 \sqrt{\pi}q_0}{2}\sum_{j=0}^\infty \frac{(-1)^j (|\underline{q}| \partial_{q_0})^{2j}A_1(q_0)}{2^{2j}j! \Gamma \left(\frac{3}{2}+j\right)}\\
\label{p2}
&&-\frac{3 \sqrt{\pi}}{4} \sum_{j=0}^\infty \frac{(-1)^j | \underline{q}|^{2j+2} \partial_{q_0}^{2j+1}A_1(q_0)}{2^{2j} j! \Gamma \left( \frac{5}{2}+j\right)},
\end{eqnarray}
and
\begin{equation}
\label{p3}
GCK_0[g_2(q_0)](q)= \frac{\sqrt{\pi}}{2} \sum_{j=0}^\infty \frac{(-1)^j | \underline{q}|^{2j} \partial_{q_0}^{2j+1}A_0(q_0)}{2^{2j}j! \Gamma \left(j+\frac{3}{2}\right)}- \frac{3 \sqrt{\pi}}{2} \sum_{j=0}^\infty \frac{(-1)^j (|\underline{q}| \partial_{q_0})^{2j}A_1(q_0)}{2^{2j}j! \Gamma \left(\frac{3}{2}+j\right)}.
\end{equation}

Thus by \eqref{p2} and \eqref{p3} we have
\begin{eqnarray}
\nonumber
GCK_0[g_1(q_0)](q)+q_0GCK_0[g_2(q_0)](q)&=&\frac{\sqrt{\pi}}{2} \sum_{j=0}^\infty \frac{(-1)^j (|\underline{q}| \partial_{q_0})^{2j}A_0(q_0)}{2^{2j}j! \Gamma \left(\frac{3}{2}+j\right)}\\
\nonumber
&&+ \frac{\sqrt{\pi}}{4} \sum_{j=0}^\infty \frac{(-1)^j (|\underline{q}| \partial_{q_0})^{2j+2} A_0(q_0)}{2^{2j}j! \Gamma \left(\frac{5}{2}+j\right)}\\
\label{scal}
	&&-\frac{3 \sqrt{\pi}}{4} \sum_{j=0}^\infty \frac{(-1)^j | \underline{q}|^{2j+2} \partial_{q_0}^{2j+1}A_1(q_0)}{2^{2j} j! \Gamma \left( \frac{5}{2}+j\right)}.
\end{eqnarray}

We now focus on the $1$-vector part of the GCK extension of the functions $g_1(q_0)$ and $g_2(q_0)$, which we denote by $GCK_1$. By the Leibintz formula, for $j \in \mathbb{N}$, we have
\begin{eqnarray*}
(|\underline{q}| \partial_{q_0})^{2j} \partial_{q_0} g_1(q_0)&=&-q_0 | \underline{q}|^{2j} \partial_{q_0}^{2j+2} A_0(q_0)-2j |\underline{q}|^{2j} \partial_{q_0}^{2j+1}A_0(q_0)+3 | \underline{q}|^{2 j} \partial_{q_0}^{2j}A_1(q_0)\\
&&+3 q_0 | \underline{q}|^{2j} \partial_{q_0}^{2j+1}A_1(q_0)+6j | \underline{q}|^{2j} \partial_{q_0}^{2j} A_1(q_0).
\end{eqnarray*}
Hence, we have
\begin{eqnarray}
	\nonumber
GCK_1[g_1(q_0)](q)&=& \frac{\sqrt{\pi}}{4} \sum_{j=0}^\infty \frac{(-1)^j (|\underline{q}| \partial_{q_0})^{2j} \partial_{q_0} g_1(q_0)}{2^{2j}j! \Gamma \left(j+\frac{5}{2}\right)}\\
\nonumber
&=&- \frac{\sqrt{\pi}q_0}{4} \sum_{j=0}^\infty \frac{(-1)^j | \underline{q}|^{2j} \partial_{q_0}^{2j+2}A_0(q_0)}{2^{2j}j! \Gamma \left(j+\frac{5}{2}\right)}+ \frac{\sqrt{\pi}}{8} \sum_{j=0}^\infty \frac{(-1)^j | \underline{q}|^{2j+2} \partial_{q_0}^{2j+3}A_0(q_0)}{2^{2j} j! \Gamma \left(j+\frac{7}{2}\right)}\\
\nonumber
&&+ \frac{3 \sqrt{\pi}}{4} \sum_{j=0}^\infty \frac{(-1)^j | \underline{q}|^{2j} \partial_{q_0}^{2j}A_1(q_0)}{2^{2j}j! \Gamma \left(j+\frac{5}{2}\right)}+ \frac{3 \sqrt{\pi} q_0}{4} \sum_{j=0}^\infty \frac{(-1)^j| \underline{q}|^{2j} \partial_{q_0}^{2j+1}A_1(q_0)}{2^{2j}j! \Gamma \left(j+\frac{5}{2}\right)}\\
\label{p4}
&&-\frac{3\sqrt{\pi}}{8} \sum_{j=0}^\infty \frac{(-1)^j | \underline{q}|^{2j+2} \partial_{q_0}^{2j+2}A_1(q_0)}{2^{2j}j! \Gamma \left(j+\frac{7}{2}\right)},
\end{eqnarray}
and
\begin{equation}
\label{p5}
GCK_1[g_2(q_0)](q)= \frac{\sqrt{\pi}}{4} \sum_{j=0}^\infty \frac{(-1)^j |\underline{q}|^{2j} \partial_{q_0}^{2j+2}A_0(q_0)}{2^{2j}j!\Gamma \left(j+\frac{5}{2}\right)}- \frac{3 \sqrt{ \pi}}{4} \sum_{j=0}^{\infty} \frac{(-1)^j | \underline{q}|^{2j} \partial_{q_0}^{2j+1}A_1(q_0)}{2^{2j} j! \Gamma \left(j+ \frac{5}{2}\right)}.
\end{equation}

Thus by \eqref{p4} and \eqref{p5} we have
\begin{eqnarray}
\nonumber
GCK_1[g_1(q_0)](q)+q_0 GCK_1[g_2(q_0)](q)&=&\frac{\sqrt{\pi}}{8} \sum_{j=0}^\infty \frac{(-1)^j | \underline{q}|^{2j+2} \partial_{q_0}^{2j+3}A_0(q_0)}{2^{2j} j! \Gamma \left(j+\frac{7}{2}\right)}\\
\nonumber
&&+ \frac{3 \sqrt{\pi}}{4} \sum_{j=0}^\infty \frac{(-1)^j | \underline{q}|^{2j} \partial_{q_0}^{2j}A_1(q_0)}{2^{2j}j! \Gamma \left(j+\frac{5}{2}\right)}\\
\label{vec}
&&-\frac{3\sqrt{\pi}}{8} \sum_{j=0}^\infty \frac{(-1)^j | \underline{q}|^{2j+2} \partial_{q_0}^{2j+2}A_1(q_0)}{2^{2j}j! \Gamma \left(j+\frac{7}{2}\right)}.
\end{eqnarray}
Hence by \eqref{scal} and \eqref{vec} we have
\begin{eqnarray*}
\mathrm{GCK}[g_1(q_0)](q) + q_0\, \mathrm{GCK}[g_2(q_0)](q)&=&GCK_0[g_1(q_0)](q)+q_0GCK_0[g_2(q_0)](q)\\
&&+ \underline{\omega} | \underline{q}| GCK_1[g_1(q_0)](q)+ \underline{\omega}q_0|\underline{q}| GCK_1[g_2(q_0)](q)\\
&=& \alpha(q)+ \underline{\omega} \beta(q),
\end{eqnarray*}
where $\alpha(q)$ and $\beta(q)$ have been introduced in \eqref{real} and in \eqref{imm}, respectively. Finally by \eqref{funct} we get the final result.

\end{proof}

The above result gives the following interesting result.

\begin{proposition}
The space of axially polyanalytic functions of order 2 is generated by
$$ \{\mathcal{Q}_{n}(q)\}_{n \in \mathbb{N}_0} \bigcup   \{q_0 \mathcal{Q}_m(q)\}_{m \in \mathbb{N}}, \quad q \in \mathbb{H},$$
where $\mathcal{Q}_{n}(q)$ and $ \mathcal{Q}_{m}(q)$ are the Clifford-Appell polynomials, see \eqref{CliffApp}.
\end{proposition}
\begin{proof}
By Proposition \ref{decop}  all the elements in the set of axially polyanalytic functions of order $2$ can be written as
$$ f(q)= GCK[f_1(q_0)](q)+q_0 GCK[f_2(q_0)](q),$$
where $f_1$ and $f_2$ are real analytic functions on the real line. By formula \eqref{rest} we have
\begin{eqnarray*}
f(q)&=& \sum_{n=0}^{\infty} GCK[q_0^n](q)a_n+q_0 \sum_{m=0}^{\infty} GCK[q_0^m](q) b_m\\
&=& \sum_{n=0}^{\infty} \mathcal{Q}_n(q)a_n+q_0 \sum_{m=0}^{\infty} \mathcal{Q}_m(q)b_m,
\end{eqnarray*}
where $ \{a_n\}_{n \geq 0}$, $ \{b_m\}_{m \geq 0} \subseteq \mathbb{H}$. This proves the result.
\end{proof}

In the final result of this section, we derive a plane wave decomposition for polyanalytic functions of order 2.

\begin{theorem}
\label{pgckwave}
Let $f(q)= \sum_{j=0}^\infty \underline{q}^j A_j(q_0)$ satisfying the same condition of Theorem \ref{pgck}. Then $f(q)$ can be decomposed as
\begin{eqnarray*}
f(q)&=& \frac{1}{2 \pi} \left( \int_{\mathbb{S}} \hbox{exp}(\langle \underline{\omega}, \underline{q} \rangle \partial_{q_0})(1- \langle \underline{\omega}, \underline{q} \rangle \underline{\omega} \partial_{q_0}) A_0(q_0)dS_{\underline{\omega}}\right)\\
&&+ \frac{3}{2 \pi} \int_{\mathbb{S}} \langle \underline{\omega}, \underline{q} \rangle \underline{\omega} \hbox{exp}(\langle \underline{\omega}, \underline{q} \rangle \underline{\omega} \partial_{q_0})A_1(q_0)dS_{\underline{\omega}},
\end{eqnarray*}
where $dS_{\underline{\omega}}$ is the area element of the sphere $\mathbb{S}$.
\end{theorem}

\begin{proof}
In order to write the GCK extension for polyanalytic functions of order~2 in terms of integrals over $\mathbb{S}$, we make use of Proposition~\ref{decop}. We set
$$
g_1(q) = f(q) - q_0 Df(q),
\qquad
g_2(q) = Df(q).
$$
By the proof of Proposition~\ref{decop}, we have
$$ g_1(q_0)=
\lim_{|\underline{q}| \to 0} g_1(q)
= A_0(q_0) - q_0 \partial_{q_0} A_0(q_0) + 3 q_0 A_1(q_0),
\qquad
g_2(q_0)=\lim_{|\underline{q}| \to 0} g_2(q)
= \partial_{q_0} A_0(q_0) - 3 A_1(q_0).
$$
Since $g_1(q)$ is axially Fueter regular (see Proposition~\ref{decof}), its GCK extension can be computed. More precisely, by Theorem~\ref{GCKplane} we obtain
$$ GCK[g_1](q)= \frac{1}{2 \pi} \int_{\mathbb{S}} \hbox{exp}( \langle \underline{\omega}, \underline{q} \rangle \underline{\omega} \partial_{q_0})g_1(q_0) dS_{\underline{\omega}} =\frac{1}{2 \pi} \int_{\mathbb{S}} \sum_{j=0}^\infty \frac{\langle \underline{\omega}, \underline{q} \rangle^j \underline{\omega}^j}{j!} \partial_{q_0}^j (g_1(q_0)).$$
We observe that by using the Leibniz rule we have
$$ \partial_{q_0}^j (g_1(q_0))= \partial_{q_0}^j A_0(q_0)-q_0 \partial_{q_0}^{j+1}A_0(q_0)-j \partial_{q_0}^jA_0(q_0)+3q_0 \partial_{q_0}^j A_1(q_0)+3j \partial_{q_0}^{j-1}A_1(q_0).$$
Thus we can write
\begin{eqnarray*}
\nonumber
GCK[g_1](q)&=& \frac{1}{2 \pi} \int_{\mathbb{S}} \sum_{j=0}^\infty \frac{\langle \underline{\omega}, \underline{q} \rangle^j \underline{\omega}^j \partial_{q_0}^j A_0(q_0)}{j!}dS_{\underline{\omega}}-\frac{q_0}{2 \pi} \int_{\mathbb{S}} \sum_{j=0}^\infty \frac{\langle \underline{\omega}, \underline{q} \rangle^j \underline{\omega}^j \partial_{q_0}^{j+1}A_0(q_0)}{j!}dS_{\underline{\omega}}\\
\nonumber
&&-\frac{1}{2 \pi} \int_{\mathbb{S}} \sum_{j=1}^{\infty} \frac{\langle \underline{\omega}, \underline{q} \rangle^j \underline{\omega}^j \partial_{q_0}^j A_0(q_0)}{(j-1)!}dS_{\underline{\omega}}+\frac{3 q_0}{2 \pi} \int_{\mathbb{S}} \sum_{j=0}^\infty \frac{\langle \underline{\omega}, \underline{q} \rangle^j \underline{\omega}^j \partial_{q_0}^j A_1(q_0)}{j!}dS_{\underline{\omega}}\\
\label{g1}
&&+ \frac{3}{2 \pi} \int_{\mathbb{S}} \sum_{j=1}^\infty \frac{\langle \underline{\omega}, \underline{q} \rangle^j \underline{\omega}^j \partial_{q_0}^{j-1}A_1(q_0)}{(j-1)!}dS_{\underline{\omega}}.
\end{eqnarray*}
Now, since the function $g_2(q)$ is axially Fueter regular (see Proposition \ref{decof}), by Theorem \ref{GCKplane} we have
\begin{eqnarray*}
\label{g2}	
GCK[g_2](q)&=& \frac{1}{2 \pi} \int_{\mathbb{S}} \hbox{exp}( \langle \underline{\omega}, \underline{q} \rangle \underline{\omega} \partial_{q_0}) g_2(q_0) dS_{\underline{\omega}}\\
\nonumber
&=& \frac{1}{2 \pi} \int_{\mathbb{S}} \sum_{j=0}^\infty \frac{\langle \underline{\omega}, \underline{q} \rangle^j \underline{\omega}^j \partial_{q_0}^{j+1}A_0(q_0)}{j!}dS_{\underline{\omega}}-\frac{3}{2 \pi} \int_{\mathbb{S}} \sum_{j=0}^\infty \frac{\langle \underline{\omega}, \underline{q} \rangle^j \underline{\omega}^j \partial_{q_0}^j A_1(q_0)}{j!}dS_{\underline{\omega}}.
\end{eqnarray*}
Now by \eqref{Pgck} we have
\begin{eqnarray*}
f(q)&=& \frac{1}{2 \pi} \int_{\mathbb{S}} \hbox{exp}( \langle \underline{\omega}, \underline{q}\rangle \underline{\omega} \partial_{q_0})A_0(q_0)  dS_{\underline{\omega}}- \frac{1}{2 \pi} \int_{\mathbb{S}}  \langle \underline{\omega}, \underline{q} \rangle \underline{\omega} \partial_{q_0} \sum_{j=0}^\infty \frac{\langle \underline{\omega}, \underline{q}\rangle^j \underline{\omega}^j \partial_{q_0}^j A_0(q_0)}{j!} dS_{\underline{\omega}} \\
&&+\frac{3}{2 \pi} \int_{\mathbb{S}}\langle \underline{\omega}, \underline{q} \rangle \underline{\omega}  \sum_{j=0}^\infty \frac{\langle \underline{\omega}, \underline{q}\rangle^j \underline{\omega}^j \partial_{q_0}^j A_1(q_0)}{j!} dS_{\underline{\omega}}  \\
&=&\frac{1}{2 \pi}  \int_{\mathbb{S}} \hbox{exp}(\langle \underline{\omega}, \underline{q} \rangle \partial_{q_0})(1- \langle \underline{\omega}, \underline{q} \rangle \underline{\omega} \partial_{q_0})A_0(q_0) dS_{\underline{\omega}}+ \frac{3}{2 \pi} \int_{\mathbb{S}} \langle \underline{\omega}, \underline{q} \rangle \underline{\omega} \hbox{exp}(\langle \underline{\omega}, \underline{q} \rangle \underline{\omega} \partial_{q_0})A_1(q_0)dS_{\underline{\omega}} .
\end{eqnarray*}

This proves the result.

\end{proof}

\section{A connection between fine structures and the GCK extension of polyanalytic functions of order 2}

In \cite{DDG}, the authors investigated the relationship between the Fueter theorem and the GCK extension. In particular, they proved that if $f$ is a left slice hyperholomorphic function, then
$$
\Delta f(q)=2\,\mathrm{GCK}\!\left[\partial_{q_0}^2f(q_0)\right](q).
$$
Subsequently, in \cite{DG}, an analogous connection was established between the operator $D$ and a suitable harmonic GCK extension.

Motivated by these results, in this section we further investigate the interplay between the Fueter construction and GCK-type extensions. More precisely, our goal is to relate the factorization scheme \eqref{SC1} to the GCK extension associated with polyanalytic functions of order $2$ introduced in the previous section.

\begin{theorem}
\label{appDbar}
Let $\Omega$ be an axially symmetric domain, and let $f(q)=\alpha(u,v)+I\,\beta(u,v)$ be a left slice hyperholomorphic function defined on $\Omega$, with $q=u+Iv$. Then we have
\begin{equation}
\label{app2}
\overline{D}f(q)=2 PGCK[2 \partial_{q_0}f(q_0), \partial_{q_0}^2f(q_0)](q), \qquad q \in \Omega_1,
\end{equation}
see \eqref{Stefano1} for the definition of the set $\Omega_1$.

\end{theorem}

\begin{proof}
By the Fueter-Sce mapping theorem (see Theorem~\ref{FMT}), $\overline{D} f(q)$ is axially polyanalytic of order~$2$. By Theorem~\ref{pgck}, such functions are uniquely determined by their restriction to the real line together with the restriction to the real line of the function obtained by applying the operator $\partial_{\underline{q}}$ to the function $f$. By \eqref{appbar} and the even-odd conditions \eqref{eo} we have
\begin{equation}
\label{appbar1}
\overline{D}f(q)= 2 \left( \partial_{v} \beta(q_0,v)+\frac{\beta(q_0,v)}{v} \right)+2 \underline{\omega} \partial_{q_0} \beta(q_0,v), \qquad v=|\underline{q}|, \quad  \underline{\omega}=\frac{\underline{q}}{|\underline{q}|}.
\end{equation}
By Lemma \ref{comp} we have
\begin{eqnarray}
	\nonumber
\lim_{|\underline{q}| \to 0} (\overline{D}f)(q)&=&2 \lim_{|\underline{q}| \to 0} \left(\sum_{j=0}^\infty \frac{(-1)^j|\underline{q}|^{2j}}{(2j)!} \partial_{q_0}^{2j+1}[\alpha(q_0,0)]+ \sum_{j=0}^\infty \frac{(-1)^j |\underline{q}|^{2j}}{(2j+1)!} \partial_{q_0}^{2j+1}[\alpha(q_0,0)]\right.\\
\nonumber
&& \left.+ \underline{\omega} \sum_{j=0}^\infty \frac{(-1)^j |\underline{q}|^{2j+1}}{(2j+1)!} \partial_{q_0}^{2j+2}[\alpha(q_0,0)] \right)\\
\nonumber
&=& 4 \partial_{q_0}[\alpha(q_0,0)]\\
\label{i11}
&=& 4\partial_{q_0} f(q_0).
\end{eqnarray}
We now focus on the second initial function of the GCK extension for polyanalytic functions of order $2$. By \eqref{appbar1}, the fact that $(-1)^j | \underline{q}|^{2j}=\underline{q}^{2j}$ and \eqref{f1} we get
\begin{eqnarray*}
- \frac{1}{3} \partial_{\underline{q}} [\overline{D}f](q)&=&-\frac{2}{3} \partial_{\underline{q}} \left( \sum_{j=0}^\infty \frac{\underline{q}^{2j}}{(2j)!} \partial_{q_0}^{2j+1}[\alpha(q_0,0)]+ \sum_{j=0}^\infty  \frac{\underline{q}^{2j}}{(2j+1)!} \partial_{q_0}^{2j+1}[\alpha(q_0,0)] \right.\\
&& \left.+ \sum_{j=0}^\infty \frac{\underline{q}^{2j+1}}{(2j+1)!} \partial_{q_0}^{2j+2}[\alpha(q_0,0)]\right)\\
&=& \frac{2}{3} \left( \sum_{j=1}^\infty \frac{2j \underline{q}^{2j-1}}{(2j)!} \partial_{q_0}^{2j+1}[\alpha(q_0,0)]+ \sum_{j=1}^\infty \frac{2j \underline{q}^{2j-1}}{(2j+1)!} \partial_{q_0}^{2j+1}[\alpha(q_0,0)] \right.\\
&& \left.+ \sum_{j=0}^\infty \frac{(2j+3) \underline{q}^{2j}}{(2j+1)!} \partial_{q_0}^{2j+2}[\alpha(q_0,0)]  \right).
\end{eqnarray*}
By taking the restriction of the above function to the real line we get
\begin{equation}
\label{i2}
- \frac{1}{3} \lim_{|\underline{q}| \to 0} \partial_{\underline{q}} [\overline{D}f](q)=2\partial_{q_0}^2 f(q_0).
\end{equation}
Therefore the result follows by \eqref{i11}, \eqref{i2} and Theorem \ref{pgck}.
\end{proof}

\begin{remark}
By combining the above result, the scheme \eqref{SC1} we can write the following diagram
\begin{center}
	\begin{tikzcd}[row sep=3em, column sep=6em]
		\mathcal{A}(\tilde{D})\otimes \mathbb{H}
\arrow[r, "S"]
		\arrow[d, "{2(2\partial_{q_0}, \partial_{q_0}^2)}" swap]
		& \mathcal{SH}(\Omega_1)
		\arrow[d, hookrightarrow, "\overline{D}"] \\
		(\mathcal{A}(\tilde{D})\otimes \mathbb{H})^2
		\arrow[r, "PGCK"]
		\arrow[d, "P_2" swap]
		& \mathcal{APA}_2(\Omega_1)
		\arrow[d, "D"]\\
		\mathcal{A}(\tilde{D})\otimes \mathbb{H}
		\arrow[r, "GCK"]
		& \mathcal{AM}(\Omega_1)
	\end{tikzcd}
\end{center}

where $S$ is the slice operator and $P_2$ is the projection of the first component.

\end{remark}

Using the above result, we can show that the application of the conjugate Fueter operator to a slice hyperholomorphic function can be expressed in terms of infinite series of $| \underline{q}| \partial_{q_0}$.

\begin{proposition}
	Let $f(q)=\alpha(u,v)+I \beta(u,v)$ be a slice hyperholomorphic functions. Then we can formally write the operator $\overline{D}$ applied to the function $f$ as
$$
		\overline{D}f(q)=2\sqrt{\pi}\left(\sum_{j=0}^\infty \frac{(-1)^j | \underline{q}|^{2j} (j+1) \partial_{q_0}^{2j+1}f(q_0)}{2^{2j}j! \Gamma \left(j+\frac{3}{2}\right)}+ \underline{q} \sum_{j=0}^\infty \frac{(-1)^j  |\underline{q}|^{2j}  \partial_{q_0}^{2j+2}f(q_0)}{2^{2j+1} j! \Gamma \left(j+\frac{3}{2}\right)} \right).
$$
\end{proposition}
\begin{proof}

We apply Theorem~\ref{appDbar}. Thus, we begin by substituting the initial functions  $2\partial_{q_0}f(q_0)$ and $\partial_{q_0}^2 f(q_0)$ into the axial components of the PGCK-extension \eqref{real} and \eqref{imm}; see \eqref{funct}. We first consider its scalar part, denoted by $\alpha(q)$, see \eqref{real}, and making a change of indexes we obtain
			\begingroup\allowdisplaybreaks
	\begin{eqnarray}
		\nonumber
		\alpha(q)&=&- \frac{3 \sqrt{\pi}}{4} \sum_{j=0}^\infty \frac{(-1)^j | \underline{q}|^{2j+2} \partial_{q_0}^{2j+3}f(q_0)}{2^{2j}j! \Gamma \left(j+\frac{5}{2}\right)}+ \frac{\sqrt{\pi}}{2} \sum_{j=0}^\infty \frac{(-1)^j |\underline{q}|^{2j+2} \partial_{q_0}^{2j+3}f(q_0)}{2^{2j}j! \Gamma \left(j+\frac{5}{2}\right)}\\
		\nonumber
		&& +\sqrt{\pi} \sum_{j=0}^\infty \frac{(-1)^j | \underline{q}|^{2j} \partial_{q_0}^{2j+1} f(q_0)}{2^{2j}j! \Gamma \left(j+\frac{3}{2}\right)}\\
		\nonumber
		&=& \sqrt{\pi} \left( \sum_{j=0}^\infty \frac{(-1)^j | \underline{q}|^{2j} \partial_{q_0}^{2j+1}f(q_0)}{2^{2j} j! \Gamma \left(j+\frac{3}{2}\right)}- \sum_{j=0}^\infty \frac{(-1)^j | \underline{q}|^{2j+2} \partial_{q_0}^{2j+3}f(q_0)}{2^{2j+2} j! \Gamma \left(j+\frac{5}{2}\right)}\right)\\
			\nonumber
		&=& \sqrt{\pi} \left( \sum_{j=0}^\infty \frac{(-1)^j | \underline{q}|^{2j} \partial_{q_0}^{2j+1}f(q_0)}{2^{2j} j! \Gamma \left(j+\frac{3}{2}\right)}+ \sum_{j=0}^\infty \frac{(-1)^j j | \underline{q}|^{2j} \partial_{q_0}^{2j+1}f(q_0)}{2^{2j} j! \Gamma \left(j+\frac{3}{2}\right)}\right)\\
		\nonumber
		&=& \sqrt{\pi} \left(\frac{\partial_{q_0} f(q_0)}{\Gamma \left(\frac{3}{2}\right)}+ \sum_{j=1}^\infty \frac{(-1)^j | \underline{q}|^{2j} (j+1) \partial_{q_0}^{2j+1}f(q_0)}{2^{2j}j! \Gamma \left(j+\frac{3}{2}\right)}\right)\\
		\label{bar1}
		&=& \sqrt{\pi}\sum_{j=0}^\infty \frac{(-1)^j | \underline{q}|^{2j} (j+1) \partial_{q_0}^{2j+1}f(q_0)}{2^{2j}j! \Gamma \left(j+\frac{3}{2}\right)}.
	\end{eqnarray}
	\endgroup
	We now focus on the 1-vector part of the GCK extension for polyanalytic functions of order~2, denoted by $\beta(q)$, see \eqref{imm}, and making a change of indexes obtain
			\begingroup\allowdisplaybreaks
	\begin{eqnarray}
		\nonumber
		\frac{\beta(q)}{|\underline{q}|}&=&- \frac{3}{2} \sqrt{\pi} \sum_{j=1}^\infty \frac{(-1)^j | \underline{q}|^{2j+2} \partial_{q_0}^{2j+4} f(q_0)}{2^{2j+2} j! \Gamma \left(j+\frac{7}{2}\right)}+ \frac{3 \sqrt{\pi}}{2} \sum_{j=0}^\infty \frac{(-1)^j | \underline{q}|^{2j} \partial_{q_0}^{2j+2} f(q_0)}{2^{2j+1}j! \Gamma \left(j+\frac{5}{2}\right)}\\
		\nonumber
		&&+ \sqrt{\pi} \sum_{j=0}^\infty \frac{(-1)^j |\underline{q}|^{2j+2} \partial_{q_0}^{2j+4}f(q_0)}{2^{2j+2}j! \Gamma \left(j+\frac{7}{2}\right)}\\
		\nonumber
		&=& \sqrt{ \pi} \left(- \sum_{j=0}^\infty \frac{(-1)^j | \underline{q}|^{2j+2} \partial_{q_0}^{2j+4}f(q_0)}{2^{2j+3} j! \Gamma \left(j+\frac{7}{2}\right)}+ \frac{3 }{2} \sum_{j=0}^\infty \frac{(-1)^j | \underline{q}|^{2j} \partial_{q_0}^{2j+2} f(q_0)}{2^{2j+1}j! \Gamma \left(j+\frac{5}{2}\right)}\right)\\
		\nonumber
		&=& \sqrt{\pi} \left(\sum_{j=1}^\infty \frac{(-1)^j | \underline{q}|^{2j} \partial_{q_0}^{2(j+1)}f(q_0)}{2^{2j+1}(j-1)! \Gamma \left(j+\frac{5}{2}\right)}+\frac{3 }{2} \sum_{j=0}^\infty \frac{(-1)^j | \underline{q}|^{2j} \partial_{q_0}^{2j+2} f(q_0)}{2^{2j+1}j! \Gamma \left(j+\frac{5}{2}\right)}\right)\\
		\nonumber
		&=& \sqrt{\pi}\left(\sum_{j=1}^\infty \frac{(-1)^j | |\underline{q}|^{2j} (2j+3) \partial_{q_0}^{2j+2}f(q_0)}{2^{2j+2} j! \Gamma \left(j+\frac{5}{2}\right)}+ \frac{3 \partial_{q_0}^2 f(q_0)}{4 \Gamma \left(\frac{5}{2}\right)}\right)\\
		\nonumber
		&=&\sqrt{\pi}\sum_{j=0}^\infty \frac{(-1)^j  |\underline{q}|^{2j} (2j+3) \partial_{q_0}^{2j+2}f(q_0)}{2^{2j+2} j! \Gamma \left(j+\frac{5}{2}\right)}\\
		\label{bar2}
		&=& \sqrt{\pi}\sum_{j=0}^\infty \frac{(-1)^j  |\underline{q}|^{2j}  \partial_{q_0}^{2j+2}f(q_0)}{2^{2j+1} j! \Gamma \left(j+\frac{3}{2}\right)}.
	\end{eqnarray}
	\endgroup
By formula \eqref{app2} and \eqref{funct} we have
\begin{equation}
\label{n1}
\overline{D} f(q)=2PGCK[2\partial_{q_0}f(q_0), \partial_{q_0}^2 f(q_0)](q) = 2\left (\alpha(q)+ \frac{\underline{q}}{|\underline{q}|} \beta(q)\right).
\end{equation}	
Therefore, the result follows by inserting \eqref{bar1} and \eqref{bar2} into \eqref{n1}.
\end{proof}

\begin{definition}
\label{poly}
Let $n \in \mathbb{N}$ and $q \in \mathbb{H}$ then we define
\begin{equation}
\label{pol}
\mathcal{P}_n(q)= (n+2) \mathcal{Q}_n(q)-q_0n \mathcal{Q}_{n-1}(q),
\end{equation}
where $\mathcal{Q}_n(q)$ are Clifford-Appell polynomials defined in \eqref{CliffApp}.
\end{definition}

\begin{lemma}
\label{poly2}
Let $n \in \mathbb{N}$ and $q \in \mathbb{H}$. The polynomials $ \mathcal{P}_n(q)$ are axially polyanalytic of order 2.
\end{lemma}
\begin{proof}
This follows from the polyanalytic decomposition (see Theorem~\ref{polydeco}) together with the fact that the Clifford–Appell polynomials are axially Fueter regular.
\end{proof}

The polynomials $\mathcal{P}_n(q)$ can be obtained via the GCK extension for polyanalytic functions of order 2.

\begin{proposition}
\label{polyex}
Let $n \geq 2$ and $q \in \mathbb{H}$. Then we have
\begin{equation}
\mathcal{P}_n(q)=PGCK[2 q_0^n, n q_0^{n-1}](q).
\end{equation}
\end{proposition}
\begin{proof}
By Lemma~\ref{poly2}, the polynomials $\mathcal{P}_n(q)$ are axially polyanalytic of order $2$. Hence, by Theorem~\ref{pgck}, it suffices to determine the functions $A_0(q_0)$ and $A_1(q_0)$. We begin by computing $A_0(q_0)$. By \eqref{rest}, we have
\begin{equation}
	\label{twores}
	A_0(q_0)= \lim_{|\underline{q}| \to 0} \mathcal{P}_n(q)=2q_0^n.
\end{equation}

To compute the second initial function $A_1(q_0)$, we first apply the operator $\partial_{\underline{q}}$ to $\mathcal{P}_{n}(q)$. To this end, we observe that the Clifford--Appell polynomials are axially Fueter regular, that is, $ D\mathcal{Q}_n(q)= (\partial_{q_0}+\partial_{\underline{q}})\mathcal{Q}_n(q)=0$.
Hence,
$$
\partial_{\underline{q}} \mathcal{Q}_n(q)=-\partial_{q_0} \mathcal{Q}_n(q)=-n \mathcal{Q}_{n-1}(q).
$$
This implies that
$$
\partial_{\underline{q}} \mathcal{P}_{n}(q)=-(n+2)n\,\mathcal{Q}_{n-1}(q)+q_0\,n(n-1)\,\mathcal{Q}_{n-2}(q).
$$
Therefore, we obtain
\begin{equation}
	\label{oneres}
	A_1(q_0)=-\frac{1}{3} \lim_{|\underline{q}| \to 0} \partial_{\underline{q}} \mathcal{P}_n(q)
	= -\frac{[n(n-1) q_0^{n-1}-n(n+2)q_0^{n-1}]}{3}
	= n q_0^{n-1}.
\end{equation}
Finally, the result follows from Theorem~\ref{pgck} together with \eqref{twores} and \eqref{oneres}.
\end{proof}

\begin{theorem}
\label{appmono}
Let $q \in \mathbb{H}$ and $n \geq 2$ then we have
\begin{equation}
\label{app4}
\overline{D}q^n=2 PGCK[2 \partial_{q_0} q_0^n, \partial_{q_0}^2 q_0^n](q)=2n \mathcal{P}_{n-1}(q).
\end{equation}
\end{theorem}
\begin{proof}
By the linearity of the GCK extension for polyanalytic functions of order $2$ and Proposition~\ref{polyex}, we obtain
\begin{equation}
	\label{secondT}
	\mathcal{P}_{n-1}(q)=PGCK\big[2\, q_0^{n-1}, (n-1)\, q_0^{n-2}\big](q)
	=\frac{PGCK\big[2\,\partial_{q_0} q_0^n, \partial_{q_0}^2 q_0^n\big](q)}{n}.
\end{equation}
By Theorem~\ref{appDbar}, we deduce that
\begin{equation}
	\label{oneT}
	\overline{D}q^n= 2\,\mathrm{PGCK}\big[2\,\partial_{q_0} q_0^n, \partial_{q_0}^2 q_0^n\big](q).
\end{equation}
The result follows by combining \eqref{secondT} and \eqref{oneT}.
\end{proof}

\section{A connection between polyanalytic functions of order 2 and Radon transform}

In \cite{CLSSmathAn}, the authors showed that another connection between the modules of slice hyperholomorphic functions and axially Fueter regular functions can be established through the dual Radon transform; see \cite{Hradon}. In the present paper, however, we require a slightly different version of this transform, namely the weighted Radon transform. This transform is a particular case of a more general class of integral transforms introduced in \cite{Gel}; see also \cite{Quinto} for further applications of the weighted Radon transform.
\\We first present the general definition of the weighted Radon transform and then adapt it to our framework. Let us consider $f: \mathbb{H} \to \mathbb{H}$. For each fixed $q_0 \in \mathbb{R}$, the weighted Radon transform of $f$ is defined by
$$R_w[f](q_0, r, \underline{\omega})
=\int_{L(\underline{\omega}, r)}
w(\underline{q},\underline{\omega})\, f(q_0, \underline{q}) \, d\sigma(\underline{q}),
\qquad \underline{\omega} \in \mathbb{S}, \quad r \in \mathbb{R},
$$
whenever the integral exists, where $w:\mathbb{R}^3 \times \mathbb{S} \to \mathbb{H}$ is a measurable weight function and $d\sigma$ denotes the Lebesgue surface measure on the hyperplane
$$
L(\underline{\omega}, r)=\{\underline{q} \in \mathbb{R}^3 : \langle \underline{q}, \underline{\omega} \rangle = r\}.
$$
 The weighted dual Radon transform of a continuous function $f$ is then defined by
	\begin{equation}
	\label{Radon}
\breve{R}_{w}[f](q_0,\underline q)=\frac{1}{2\pi}\int_{\mathbb S}w(\underline q, \underline\omega)f\bigl(q_0,\langle\underline q,\underline\omega\rangle\underline\omega\bigr)\,dS_{\underline\omega}.
	\end{equation}
where $dS_{\underline{\omega}}$ is the area element of the sphere $ \mathbb{S}$.

\begin{remark}
By taking $w(\underline{q},\underline{\omega})=1$ in the above construction, we recover the classical Radon transform and its dual transform.
\end{remark}

\begin{remark}
For convenience, in this section we will use the notation $\breve{R}_w[f](q_0,\underline{q})$ instead of $\breve{R}_w[f](q)$.
\end{remark}
In this section, our aim is to establish a connection between the space of slice hyperholomorphic functions and the space $\mathcal{APA}_{2, \overline{D}}^L(\Omega)$, defined in \eqref{spacepoly}, via a Radon-type transform.

\begin{definition}
\label{pradon}
Let $f: \mathbb{H} \to \mathbb{H}$. For $\underline{\omega} \in \mathbb{S}$ we define the dual $\mathcal{P}$-Radon transform as
\begin{equation}
\label{Pradon}
	\breve{R}_{\mathcal{P}}[f](q_0, \underline{q})= \frac{1}{2 \pi} \int_{\mathbb{S}} (2+ \langle \underline{\omega}, \underline{q} \rangle \underline{\omega} \partial_{q_0})f(q_0, \langle \underline{\omega}, \underline{q} \rangle \underline{\omega})  dS_{\underline{\omega}}.
\end{equation}

\end{definition}

\begin{remark}
By linearity of the integral and using \eqref{Radon}, we have the following representation of the dual $\mathcal{P}$-Radon transform:
$$\breve{R}_{\mathcal{P}}[f](q_0,\underline{q})=2
\breve{R}_1\big[f\big](q_0,\underline{q})+\breve{R}_w\!\left[\partial_{q_0} f\right](q_0,\underline{q}),$$
where $w(\underline{q}, \underline{\omega})=\langle \underline{\omega}, \underline{q} \rangle \underline{\omega}$.
\end{remark}

\begin{theorem}
\label{barR}
Let $\Omega$ be an axially symmetric domain, and let $f(q_0,\underline{q})=\alpha(u,v)+I\,\beta(u,v)$ be a left slice hyperholomorphic function defined on $\Omega$. Then we have
\begin{equation}
\overline{D}f(q)= 2 \breve{R}_{\mathcal{P}}[\partial_{q_0}f](q_0, \underline{q}),
\end{equation}
where $\breve{R}_{\mathcal{P}}[f]$ si defined in \eqref{Pradon}.
\end{theorem}
\begin{proof}
The result is a direct consequence of Theorem \ref{appDbar} and Theorem \ref{pgckwave}:
\begin{eqnarray}
\nonumber
\overline{D}f(q)&=&2 PGCK[2 \partial_{q_0}f, \partial_{q_0}^2f](q)\\
\nonumber
&=& \frac{1}{\pi} \left[2 \int_{\mathbb{S}} \hbox{exp}(\langle \underline{\omega}, \underline{q} \rangle \underline{\omega}\partial_{q_0})(1- \langle \underline{\omega} ,\underline{q} \rangle  \underline{\omega}\partial_{q_0}) \partial_{q_0} f(q_0)dS_{\underline{\omega}} \right.\\
\nonumber
&& \left. + 3 \int_{\mathbb{S}}  \hbox{exp}(\langle \underline{\omega}, \underline{q} \rangle \underline{\omega}\partial_{q_0})\langle \underline{\omega}, \underline{q} \rangle \underline{\omega} \partial_{q_0}^2 f(q_0) dS_{\underline{\omega}} \right]\\
\label{radon1}
&=& \frac{1}{\pi} \int_{\mathbb{S}} \hbox{exp}(\langle \underline{\omega}, \underline{q} \rangle \underline{\omega} \partial_{q_0}) (2+ \langle \underline{\omega},  \underline{q} \rangle \partial_{q_0}) (\partial_{q_0}f(q_0)) dS_{\underline{\omega}}.
\end{eqnarray}
 Now we observe that
\begin{equation}
\label{radon2}
\hbox{exp}(\langle \underline{\omega}, \underline{q} \rangle \underline{q} \partial_{q_0})(\partial_{q_0}f(q))= \sum_{j=0}^{\infty} \frac{\langle \underline{\omega}, \underline{q} \rangle^j \underline{\omega}^j}{j!} \partial_{q_0}^{j+1}[f](q_0)=f'(q_0+\langle \underline{\omega}, \underline{q} \rangle \underline{\omega}).
\end{equation}

Finally, the result follows by substituting \eqref{radon2} into \eqref{radon1} and applying Definition \ref{pradon}.
\end{proof}

By using the $ \mathcal{P}$-Radon type transform we can have the following properties of the polynomials $\mathcal{P}_n(q)$, see Definition \ref{poly}.

\begin{corollary}
Let $m \in \mathbb{N}$ then we have
$$ \mathcal{P}_m(q)= \breve{R}_{\mathcal{P}}[q^m](q_0, \underline{q}).$$
\end{corollary}
\begin{proof}
The result follows by combining Theorem~\ref{barR} and formula \eqref{app4}, indeed
$$ \mathcal{P}_m(q)= \frac{1}{2(m+1)} \overline{D}q^{m+1}=\frac{1}{m+1}\breve{R}_{\mathcal{P}}[\partial_{q_0} q^{m+1}](q_0, \underline{q})= \breve{R}_{\mathcal{P}}[q^m](q_0, \underline{q}).$$
\end{proof}

In order to describe the mapping properties of the $\mathcal{P}$-Radon type transform, we need the following result.

\begin{lemma}
	\label{inver}
	Let $\Omega$ be an axially symmetric set and $\Omega_1$ being the set defined in \eqref{Stefano1}. Let $f \in \mathcal{SH}_L(\Omega)$ . Then there exists $g \in \mathcal{SH}_L(\Omega_1)$ such that $ \partial_{q_0}g(q)=f(q)$.
\end{lemma}
\begin{proof}
Let
$$
g(q_0):= \int_0^{q_0} f(s_0)\, ds_0.
$$
Since, by hypothesis, $f \in \mathcal{SH}_L(\Omega)$, Theorem~\ref{iso} implies that $g(q_0)$ is a real-analytic function admitting a unique isomorphic extension. Hence, using the slice regular extension $S_1$ (see \eqref{regex}), we can extend $g(q_0)$ to a slice hyperholomorphic function, namely
$$
g(q)= \sum_{j=0}^{\infty} \frac{\underline{q}^j}{j!}\, \partial_{q_0}^j g(q_0).
$$
We observe that $\partial_{q_0} g(q)$ is slice hyperholomorphic. Moreover, by hypothesis we have
$$
\partial_{q_0} g(q)\big|_{\underline{q}=0}= f(q)\big|_{\underline{q}=0}.
$$
Thus, by the uniqueness of the slice regular extension, we conclude that $ \partial_{q_0} g(q)=f(q)$.
\end{proof}

\begin{theorem}
Let $\Omega$ be an axially symmetric set. The dual $\mathcal{P}$-Radon transform maps the set of slice hyperholomorphic functions to the set $ \mathcal{APA}_{2, \overline{D}}(\Omega)$.
\end{theorem}
\begin{proof}
By Lemma \ref{inver} there exists $g \in \mathcal{SH}_L(\Omega_1)$ such that $ \partial_{q_0}g(q)=f(q)$. Thus, by Theorem \ref{barR}  we have
$$ \breve{R}_{\mathcal{P}}[f](q_0, \underline{q})=\breve{R}_{\mathcal{P}}[\partial_{q_0}g](q_0, \underline{q})=\frac{\overline{D}g(q)}{2} \in\mathcal{APA}_{2, \overline{D}}(\Omega).$$
This proves the result.
\end{proof}

\section{On the invertibility of $\overline{D}$ applied to slice hyperholomorphic functions}

In this section, we aim to determine whether the space obtained as the image of $\overline{D}$, namely $\mathcal{APA}^L_{2, \overline{D}}(\Omega)$ (see \eqref{spacepoly}), coincides with the space $\mathcal{APA}_2(\Omega)$; see Definition~\ref{polyorder2}. In particular, we want to investigate whether
\begin{equation}
\label{mapfine}
\overline{D}: \mathcal{SH}_L(\Omega) \longrightarrow \mathcal{APA}_2^L(\Omega)
\end{equation}
is surjective. In \cite{CSSOinverse}, it was proved that the map
$$
\Delta_4: \mathcal{SM}_L(\Omega) \to \mathcal{AM}_L(\Omega)
$$
is surjective. However, by factorizing the operator $\Delta_4$, this surjectivity property is no longer preserved, as can be seen in the following example.

\begin{lemma}
\label{inclusion}
Let $\Omega \subseteq \mathbb H$ be an axially symmetric domain such that $\Omega\cap\mathbb R\neq\emptyset$
The space function $\mathcal{APA}^L_{2, \overline{D}}(\Omega)$ is properly included in $\mathcal{APA}_2(\Omega)$, i.e. we have
$$ \mathcal{APA}^L_{2, \overline{D}}(\Omega) \subsetneq \mathcal{APA}_2(\Omega).$$
\end{lemma}
\begin{proof}
The inclusion follows trivially from the definitions of the spaces. To see that the inclusion is proper, we consider the following example
$$ g(q)=\underline{q}.$$
The function is left axially polyanalytic of order 2, indeed
$$ D^2 g(q)= \frac{1}{2} D^2(q-\bar{q})=0,$$
thus $g \in \mathcal{APA}_2(\Omega)$. Now we show that $g \notin \mathcal{APA}^L_{2, \overline{D}}(\Omega)$. We prove the previous statement by contradiction. Suppose that there exists a slice hyperholomorphic function $g_1$ such that
\begin{equation}
\label{appp}
\overline D g_1(q) = \underline{q}.
\end{equation}
We restrict to the real axis. Since $\overline D g_1(q)|_{\underline{q}=0}=4\,\partial_{q_0} g_1(q_0)$, see \cite[Thm. 4.10]{CDS2025}. By restriction \ref{appp} to the real line we obtain
$$
4\,\partial_{q_0} g_1(q_0) = 0,
$$
which implies that $ g_1$ is constant on $\mathbb{R} $. By the uniqueness of slice hyperholomorphic continuation from the real axis, it follows that $g_1$ is constant on $\mathbb{H}$; that is,
$$
g_1(q) = k \quad \text{for all } q \in \mathbb{H}.
$$

However, this condition is incompatible with the identity \( \overline D g_1 = \underline{q} \). Therefore, no such slice hyperholomoprhic primitive can exist. This contradiction shows that \( \underline{q} \) does not belong to the image of the operator \( \overline D \) acting on the class of slice hyperholomorphic functions.
\end{proof}

In this paper, our aim is to determine a vector subspace $\mathcal{APA}^L_{2,0}(\Omega) \subseteq \mathcal{APA}^L_2(\Omega)$ such that
\begin{equation}\label{decomp_1}
	\mathcal{APA}^L_2(\Omega
	)
	=
	\mathcal{APA}^L_{2, \overline{D}}(\Omega)
	\oplus \mathcal{APA}^L_{2,0}(\Omega) ,
\end{equation}

In the following result, we provide a characterization of the subspace $\mathcal{APA}^L_{2,0}(\Omega) $.

 \begin{theorem}\label{t2}
 	Let $\Omega \subseteq \mathbb H$ be an axially symmetric domain such that $\tilde\Omega:=\Omega \cap\mathbb R\neq \emptyset$.
 	Every function $f\in \mathcal{APA}^L_2(\Omega)$ in a neighborhood $\Omega_1$ (see \eqref{Stefano1}) of $\tilde \Omega$ can be uniquely decomposed as
 	\begin{equation}\label{Stefano2}
 	f(q)=\overline D g_1(q) + g_2(q), \quad\textrm{for $q\in\Omega_1$}
 	\end{equation}
 	where $g_1\in\mathcal{SH}_L( \Omega_1)$ (determined up to an additive constant) and $g_2\in \mathcal{APA}^L_2(\Omega_1)$ satisfies
 	\[
 	g_2(q)|_{\underline q=0}=0.
 	\]
 \end{theorem}

\begin{proof}
	We first construct a slice hyperholomorphic function $g_1\in\mathcal{SH}_L(\Omega_1)$ such that ${(\overline Dg_1)}_{|_{\unq=0}}=f_{|_{\unq=0}}$. Since $g_1 \in \mathcal{SH}_L( \Omega_1)$ by \eqref{exp} we can write the following expansion in series
	\begin{equation}
\label{sstar}
g_1(q)=\sum_{\ell\geq 0} \frac{1}{\ell !} \unq^\ell \partial^\ell_{q_0} g_1(q_0),
	\end{equation}
	where the convergence is uniform on any compact subset of $\Omega_1$ . In particular, applying $\overline D=\partial_{q_0}-\partial_{\underline q}$ to the previous summation and restricting to the real axis, see \cite[Thm. 4.10]{CDS2025}, we have
	$$ f(q)|_{\underline{q}=0}=\overline D(g_1)(q)_{|_{\unq=0}}= 4\partial_{q_0} g_1(q_0). $$
	Thus we have
	$$ g_1(q_0):=\frac 14 \int_c^{q_0} f(t)dt, $$
	where $c\in \tilde \Omega$. Since $f|_{\mathbb R}$ is real analytic, then $g_1$ is also real analytic. So we can extend $g_1(q_0)$ to a slice hyperholomorphic function using formula \eqref{sstar}. Finally, we define
	\[
	g_2(q):=f(q)-\overline D g_1(q).
	\]
The function $g_2$ is axially polyanalytic of order $2$. Its axiality follows from the hypothesis and Lemma~\ref{apDbarD}, while its polyanalyticity of order $2$ is a consequence of the hypothesis and the Fueter theorem (see Theorem~\ref{FMT}). Indeed,
$$
D^2 g_2(q) = D^2 f(q) - D \Delta_4 g_1(q) = 0.
$$
Moreover, by construction we have
	$$g_2(q)|_{\underline q=0}=0.$$
This proves the result.
\end{proof}

\begin{remark}
The decomposition in \eqref{Stefano2} can be extended to the whole $\Omega$ by means of the integral representation given in Theorem \ref{t3}.
\end{remark}

\begin{definition}
\label{Dprim}
	Let $\Omega_1$ as in \eqref{Stefano1}. The function $g_1\in \mathcal{SH}_L(\Omega_1)$ is called the weak $\overline D$-primitive of $f\in \mathcal{APA}_2^L(\Omega_1)$ if $(\overline D{g_1})|_{\underline q=0}=f |_{\underline q=0}$. The name weak is due to the fact that, unlike the $\Delta$-primitive (see \cite{CSSOinverse}), we consider only the restriction to the real line.
\end{definition}

\begin{remark}
According to Theorem \ref{t2}, the vector space for which the decomposition in \eqref{decomp_1} holds can be written as:
\begin{equation}
	\label{newpsace}
\mathcal{APA}^L_{2,0}(\Omega) =\{g_2\in \mathcal{APA}^L_2(\Omega): \, g_2|_{\underline q=0}=0\}\subseteq \mathcal{APA}^L_2(\Omega).
\end{equation}
\end{remark}

\begin{remark}\label{r1}
The space $\mathcal{APA}^L_{2,0}(\Omega) $ introduced in \eqref{decomp_1} is essential to characterize the entire class of axially polyanalytic functions of order $2$. Indeed, Lemma~\ref{inclusion} establishes the existence of a function $f(q)=\underline{q}$ which is polyanalytic of order $2$, but does not belong to the image of the operator $\overline{D}$ acting on the class of slice hyperholomorphic functions.
\end{remark}

\begin{proposition}
\label{zerobasis}
Let $\Omega_1$ be as \eqref{Stefano1}. The space $\mathcal{APA}^L_{2,0}(\Omega_1)$ is generated by polynomials of the form
\begin{equation}
\label{bh}
p_n(q):= \mathcal{Q}_{n+1}(q)-q_0 \mathcal{Q}_n(q),
\end{equation}
where $ \mathcal{Q}_n(q)$ are the Clifford-Appell polynomials defined in \eqref{CliffApp}.
\end{proposition}
\begin{proof}
By Theorem~\ref{Pgck} and \eqref{newpsace}, to determine a basis for the space $\mathcal{APA}^L_{2,0}(\Omega_1)$, it is enough to compute $PGCK[0, f_1(q_0)](q)$, where $f_1(q_0)$ is an analytic function. By the linearity of the GCK extension for polyanalytic functions of order $2$ and Theorem~\ref{decop}, we have
\begin{align*}
	PGCK[0, f_1(q_0)](q)
	&= \sum_{n=0}^\infty PGCK[0, q_0^n](q) a_n \\
	&= \sum_{n=0}^\infty GCK[h_1(q_0)](q) a_n + q_0 \sum_{n=0}^{\infty} GCK[h_2(q_0)](q) a_n,
\end{align*}
where $\{a_n\}_{n \geq 0} \subseteq \mathbb{H}$, and $h_1(q_0) := 3 q_0^{n+1}$ and $h_2(q_0) := -3 q_0^n$ are obtained from \eqref{res1} and \eqref{res2}, respectively. Finally, by \eqref{rest}, we get
$$	PGCK[0, f_1(q_0)](q) =3 \sum_{n=0}^{\infty} \mathcal{Q}_{n+1}(q)a_n-3 q_0 \sum_{n=0}^\infty \mathcal{Q}_n(q)a_n= 3 \sum_{n=0}^{\infty} p_n(q)a_n.$$
This proves the result.
\end{proof}

From the above result, together with Theorem~\ref{t2} and Theorem~\ref{appmono}, we obtain the following theorem.

\begin{theorem}
	Let $\Omega$ and $\Omega_1$ be as in Theorem~\ref{t2}. Then any function $f \in \mathcal{APA}_2^L(\Omega)$ can be written as
	$$
	f(q)= \sum_{m=0}^\infty \mathcal{P}_m(q)a_m+\sum_{n=0}^{\infty} p_n(q) b_n, \qquad  \{b_n\}_{n \geq 0}, \, \{a_m\}_{m \geq 0} \subseteq \mathbb{H},
	$$
where the polynomials $\mathcal{P}_m(q)$ are defined in \eqref{pol}.
\end{theorem}

In what follows, we provide an example illustrating the above theorem.

\begin{example}
Let us consider the function
$$
f(q)= \mathcal{Q}_2(q)+q_0 \mathcal{Q}_1(q),
$$
where $\mathcal{Q}_2(q)$ and $\mathcal{Q}_1(q)$ are the Clifford--Appell polynomials. By Theorem~\ref{polydeco} and the fact that the Clifford--Appell polynomials are axially Fueter regular, it follows that the function $f$ is polyanalytic of order $2$.
	
By the definition of the polynomials $\mathcal{P}_2(q)$ and $p_1(q)$ (see \eqref{bh}), we have
$$
\mathcal{Q}_2(q)= \frac{\mathcal{P}_2(q)}{4}+ \frac{q_0 \mathcal{Q}_1(q)}{2},
$$
$$
q_0 \mathcal{Q}_1(q)=\mathcal{Q}_2(q)-p_1(q).
$$
Substituting the second identity into the first, we obtain
$$
\mathcal{Q}_2(q)=\frac{\mathcal{P}_2(q)}{4}- \frac{p_1(q)}{2}+\frac{\mathcal{Q}_2(q)}{2},
$$
	and hence
$$
\mathcal{Q}_2(q)= \frac{\mathcal{P}_2(q)}{2}-p_1(q).
$$
Consequently,
$$
q_0 \mathcal{Q}_1(q)= \frac{\mathcal{P}_2(q)}{2}-2p_1(q).
$$
Therefore, we can write
$$
f(q)=\mathcal{Q}_2(q)+q_0 \mathcal{Q}_1(q)= \mathcal{P}_2(q)-3p_1(q).
$$
\end{example}

In the following results, we aim to derive an explicit integral representation of the weak $\overline{D}$-primitive, see Definition \ref{Dprim}. To this end, we first introduce appropriate kernels and analyze their restriction to the real axis.

\begin{definition} \label{d1} Let $E(q-s)$ be the axially Fueter regular Cauchy kernel (see \eqref{kernelC}) with $q=q_0+\unq\in\mathbb H$, $s=s_0+\uns\in\mathbb H$ and $\unq=\underline{\omega} \rho$ where $\underline{\omega}\in\mathbb S$ and $\rho=1$. We define the kernels
	
	\begin{align*}
		\mathcal{N}^{1,+}\left (s \right)&:=\int_{\mathbb S} E\left( s - \underline{\omega} \right) dS_{\underline{\omega}}\quad\quad\quad\quad\quad
		\mathcal{N}^{1,-}\left(s \right):=\int_{\mathbb S} E\left( s - \underline{\omega} \right) \underline{\omega} dS_{\underline{\omega}}\\
		\mathcal{N}^{2,+}\left(s \right) &:= \int_{\mathbb S} s_0 E\left( s - \underline{\omega} \right) dS_{\underline{\omega}}\quad
		\mathcal{N}^{2,-}\left(s  \right) :=\int_{\mathbb S}  s_0 E\left( s - \underline{\omega} \right)  \underline{\omega} dS_{\underline{\omega}},
	\end{align*}
	where $dS_{\underline{\omega}}$ is the scalar element of surface area of $\mathbb S$.
\end{definition}

\begin{proposition}\label{p1} The restrictions of the kernels introduced in Definition \ref{d1} to the real axis $\uns=0$ are given by
	\begin{align*}
		\mathcal N^{1,+}(s)_{|_{\uns=0}}= \frac{4}\pi \frac{s_0}{(s_0^2+1)^{2}} & \quad\quad
		\mathcal N^{1,-}(s)_{|_{\uns=0}}=\frac{4}\pi \frac{1}{(s_0^2+1)^{2}}\\
		\mathcal N^{2,+}(s)_{|_{\uns=0}}=\frac{4}\pi \frac{s_0^2}{(s_0^2+1)^{2}} &\quad\quad
		\mathcal N^{2,-}(s)_{|_{\uns=0}}=\mathcal N^{1,+}(s)_{|_{\uns=0}}.
	\end{align*}
\end{proposition}
\begin{proof}
	The proof for the kernels $\mathcal N^{1,\pm}$ coincides with the proof of Theorem 3.6 in \cite{CSSinve}.
	The proof for the kernels $\mathcal N^{2,\pm}$ follows the same line of argument as in the previous case, once it is observed that the factor $s_0$ can be taken outside the integral sign.
	Both cases rely on the Funk-Hecke Theorem (see  \cite{HOCH}).
\end{proof}
\begin{proposition}\label{restriction}
	The weak $\overline D$-primitives for the kernels introduced in Definition \ref{d1} are given by
	\begin{align*}
		\mathcal W^{+,1}(s)=\frac{-2}{\pi} \frac{1}{1+s^2} &\quad\quad \mathcal W^{-,1}(s)=\frac{2}{\pi} \left(\arctan(s) + \frac{s}{1+s^2} \right) \\
		\mathcal W^{+,2}(s)=\frac{2}{\pi} \left(\arctan(s) - \frac{s}{1+s^2} \right) &\quad\quad \mathcal W^{-,2}(s)=\mathcal W^{+,1}(s).
	\end{align*}
\end{proposition}
\begin{proof}
	We seek the left slice hyperholomorphic functions $\mathcal W^{1,\pm}(s)$ and $\mathcal W^{2,\pm}(s)$ satisfying
	$$ (\overline D \mathcal W^{\pm,1,2}(s))_{|_{\uns=0}}= \mathcal N^{\pm, 1,2}(s)_{|_{\uns=0}}$$
	where the restrictions of the kernels $\mathcal N^{\pm, 1,2}(s)$ to the real axis are determined in Proposition \ref{p1}. Moreover, since by the definition $\mathcal W^{\pm, 1,2}(s)$ are slice hyperholomorphic functions, as in the proof of Theorem \ref{t2}, one has $(\overline D \mathcal W^{\pm, 1,2}(s))_{|_{\uns=0}}=4\partial_{s_0} W^{\pm, 1,2}(s_0)$. These relations determine $(\mathcal W^{\pm, 1,2}(s))_{|_{\uns=0}}$ up to an additive constant, namely
	$$ \mathcal W^{\pm,1,2}(s_0) =\frac 14 \int_c^{s_0} \mathcal N^{\pm 1,2}(t) \, dt, $$
	for $c\in U\cap\mathbb R$. Because the slice hyperholomorphic extension from the real axis is unique, $\mathcal W^{\pm, 1,2}(s)$ are uniquely determined by its restriction to the real axis. A direct computation of the previous integrals, for the different cases yields, the explicit expressions for $\mathcal{W}^{\pm, 1,2}(s)$ stated in the theorem.
\end{proof}
\begin{remark}
We finally note that the kernels $\mathcal{W}^{\pm 1,2}$ are defined not only in a neighborhood of the real axis but in their natural domain of the quaternions.
\end{remark}

We are now ready to provide an integral representation of the weak $\overline{D}$-primitive, see Definition \ref{Dprim}. To this end, we follow the arguments used in the proof of \cite[Thm.~4.2]{CSSinve}.

\begin{theorem}\label{t3}
	Let $\Omega, U\subseteq \mathbb H$ be domains such that $U$ is axially symmetric and it satisfies $U\subset \Omega$ and $U\cap\mathbb R\neq \emptyset$. For any $f\in \mathcal{APA}^L_2(\Omega)$, the weak $\overline D$-primitive of $f$ is given by
	\begin{align*}
		g_1(s)& =-\int_{\Gamma} \mathcal W^{+,1} \left( \frac{s-q_0}{\rho} \right)  (Ad\rho +Bdq_0)-\int_{\Gamma}  \mathcal W^{-,1} \left( \frac{s-q_0}{\rho} \right)   (-Adq_0 +Bd\rho)\\
		&-\int_{\Gamma}  \mathcal W^{+,2} \left( \frac{s-q_0}{\rho} \right) \rho (Cd\rho +Fdq_0)-\int_{\Gamma}  \mathcal W^{-,2} \left( \frac{s-q_0}{\rho} \right) \rho (-Cdq_0 +Fd\rho).
	\end{align*}
	where $\mathcal W^{\pm,1,2}$ are given in Proposition \ref{restriction}, $\Gamma$ is the $C^1$ curve given as the intersection of $U$ with an arbitrary complex slice through the real axis, $A(q_0,\rho),\, B(q_0,\rho),\, C(q_0,\rho)$ and $F(q_0,\rho)$, are four quaternionic valued functions such that
	$$ f(q)=f(q_0+\omega\rho)=A(q_0,\rho)+\underline{\omega} B(q_0,\rho),\quad Df(q)=Df(q_0+\omega\rho)=C(q_0,\rho)+\underline{\omega} F(q_0,\rho) $$
	with $C(q_0,\rho):=\partial_{q_0} A(q_0,\rho)-\partial_\rho B(q_0,\rho)-\frac 2r B(q_0,\rho)$ and $F(q_0,\rho):=\partial_{q_0}B(q_0,\rho)+\partial_\rho A(q_0,\rho)$.
\end{theorem}

\begin{proof}

	 We consider two quaternionic variables $s,\,q$, together with their decompositions: $q=q_0+\underline{\omega} \rho$, $s=s_0+Ir$, with $\underline{\omega},\, I\in\mathbb S$. By the assumptions on the domain $U$, its boundary can be written as $\partial U=\{q_0+\underline{\omega}\rho\in \mathbb H:\, \underline{\omega}\in\mathbb S, \, (q_0,\rho) \in\Gamma \}$. We may parametrize $\Gamma$ by arc length as $\ell\mapsto q_0(\ell)+\underline{\omega}\rho(\ell)$. Thus, its tangent vector is:
	$$t(\ell):=\dot{q}_0(\ell)+\underline{\omega} \dot{\rho}(\ell),$$
	while its outward normal is:
	$$n(\ell):=\dot \rho(\ell)-\underline{\omega} \dot{q}_0(\ell).$$
	We begin with the integral representation for axially polyanalytic functions of order 2 (see Theorem \ref{t1}):
	\begin{equation}\label{i1}
		f(s)=\int_{\partial U} E(q-s) \, d\sigma_q f(q)- \int_{\partial U} (q_0-s_0)E(q-s) \, d\sigma_q D f(q).
	\end{equation}
	Since $\partial U=\Gamma\times \mathbb S$, thus
	$$ d\sigma_q=n(\ell)\rho^2\, d\ell\, dS_{\underline{\omega}}=(d\rho(\ell)-\underline{\omega} dq_0(\ell))\rho^2 dS_{\underline{\omega}} $$
	where $dS_{\underline{\omega}}$ denotes the surface measure on $\mathbb S$. Since $f$ is of axial type by Lemma~\ref{apDbarD}, there exist four quaternion-valued functions $A(q_0,\rho)$, $B(q_0,\rho)$, $C(q_0,\rho)$, and $F(q_0,\rho)$ such that
	$$ f(q)=f(q_0+\omega\rho)=A(q_0,\rho)+\underline{\omega} B(q_0,\rho),\quad Df(q)=Df(q_0+\omega\rho)=C(q_0,\rho)+\underline{\omega}  F(q_0,\rho) $$
	where $C(q_0,\rho):=\partial_{q_0} A(q_0,\rho)-\partial_\rho B(q_0,\rho)-\frac 2r B(q_0,\rho)$ and $F(q_0,\rho):=\partial_{q_0}B(q_0,\rho)+\partial_\rho A(q_0,\rho)$, see Lemma \ref{apDbarD}. Substituting these expressions in \eqref{i1} yields
	\begin{align}\label{e2}
		& f(s)  =\int_{\Gamma}\int_{\mathbb S} E(q-s) dS_{\underline{\omega}}\rho^2(Ad\rho +Bdq_0) +\int_{\Gamma}\int_{\mathbb S} E(q-s)  \underline{\omega} dS_{\underline{\omega}}\rho^2(-Adq_0 +Bd\rho)\\
		&-\int_{\Gamma}\int_{\mathbb S} (q_0-s_0)E(q-s) dS_{\underline{\omega}}\rho^2(Cd\rho +Fdq_0)-\int_{\Gamma}\int_{\mathbb S} (q_0-s_0) E(q-s)  \underline{\omega} dS_{\underline{\omega}}\rho^2 (-Cdq_0 +Fd\rho).\nonumber
	\end{align}
We are omitting the variables in the functions $A$, $B$, $C$, and $F$ in order not to burden the notation.
We now observe that, since $E(tq)=t^{-3}E(q)$ for $t>0$, it follows that
	$$ E(q-s)=E(q_0+\underline{\omega} \rho -s_0-rI)=-\rho^{-3} E\left( \frac{s_0-q_0}{\rho}+\frac r\rho I - \underline{\omega} \right)$$
	and this implies that
	$$ (q_0-s_0)E(q-s)=\rho^{-3} (s_0-q_0) E\left( \frac{s_0-q_0}{\rho}+\frac r\rho I - \underline{\omega} \right).$$
	Thus, substituting these identities in \eqref{e2} we obtain
	\begin{align}\label{e3}
		f(s) & =-\int_{\Gamma} \int_{\mathbb S}  E\left( \frac{s_0-q_0}{\rho}+\frac r\rho I - \underline{\omega} \right) dS_{\underline{\omega}} \rho^{-1} (Ad\rho +Bdq_0)\nonumber\\
		&-\int_{\Gamma}\int_{\mathbb S}  E\left( \frac{s_0-q_0}{\rho}+\frac r\rho I - \underline{\omega} \right)  \underline{\omega} dS_{\underline{\omega}} \rho^{-1} (-Adq_0 +Bd\rho)\\
		&-\int_{\Gamma} \int_{\mathbb S}  \frac{(s_0-q_0)}{\rho} E\left( \frac{s_0-q_0}{\rho}+\frac r\rho I - \underline{\omega} \right) dS_{\underline{\omega}}  (Cd\rho +Fdq_0)\nonumber\\
		&-\int_{\Gamma}\int_{\mathbb S}  \frac{(s_0-q_0)}{\rho} E\left( \frac{s_0-q_0}{\rho}+\frac r\rho I - \underline{\omega} \right)  \underline{\omega} dS_{\underline{\omega}} (-Cdq_0 +Fd\rho). \nonumber
	\end{align}
Using the kernels $\mathcal{N}^{\pm, 1,2}$ given in Definition~\ref{d1}, and since $s=s_0+Ir$, we can rewrite equation \eqref{e3} as
	\begin{align}\label{e4}
		f(s) & =-\int_{\Gamma} \mathcal N^{+,1} \left( \frac{s-q_0}{\rho} \right) \rho^{-1} (Ad\rho +Bdq_0)-\int_{\Gamma} \mathcal N^{-,1} \left( \frac{s-q_0}{\rho} \right) \rho^{-1} (-Adq_0 +Bd\rho)\\
		&-\int_{\Gamma} \mathcal N^{+,2} \left( \frac{s-q_0}{\rho} \right)   (Cd\rho +Fdq_0)-\int_{\Gamma} \mathcal N^{-,2} \left( \frac{s-q_0}{\rho} \right)   (-Cdq_0 +Fd\rho). \nonumber
	\end{align}
	Thus, by Proposition \ref{restriction}, and the definition of weak $\overline{D}$-primitive, denoted here by $g_1$, it follows that
	\begin{align*}
		\overline D g_1(s)& =-\int_{\Gamma} \overline{D}_{s'} \mathcal W^{+,1} \left( \frac{s-q_0}{\rho} \right) \rho^{-1} (Ad\rho +Bdq_0)-\int_{\Gamma} \overline{D}_{s'} \mathcal W^{-,1} \left( \frac{s-q_0}{\rho} \right)  \rho^{-1} (-Adq_0 +Bd\rho)\\
		&-\int_{\Gamma} \overline{D}_{s'} \mathcal W^{+,2} \left( \frac{s-q_0}{\rho} \right)  (Cd\rho +Fdq_0)-\int_{\Gamma} \overline{D}_{s'} \mathcal W^{-,2} \left( \frac{s-q_0}{\rho} \right)  (-Cdq_0 +Fd\rho),
	\end{align*}
	where $s'=\frac{s-q_0}{\rho}$. By using the identity $\overline D_{s'}=\rho \overline D_s$ we deduce that
	\begin{align*}
		g_1(s)& =-\int_{\Gamma} \mathcal W^{+,1} \left( \frac{s-q_0}{\rho} \right)  (Ad\rho +Bdq_0)-\int_{\Gamma}  \mathcal W^{-,1} \left( \frac{s-q_0}{\rho} \right)   (-Adq_0 +Bd\rho)\\
		&-\int_{\Gamma}  \mathcal W^{+,2} \left( \frac{s-q_0}{\rho} \right) \rho (Cd\rho +Fdq_0)-\int_{\Gamma}  \mathcal W^{-,2} \left( \frac{s-q_0}{\rho} \right) \rho (-Cdq_0 +Fd\rho).
	\end{align*}
This function is slice hyperholomorphic since the functions $\mathcal W^{\pm,1,2}$ are slice hyperholomorphic by construction. So $g_1$ is the weak $\overline{D}$ primitive of $f$.
\end{proof}

\section{On the invertibility of $D$ applied to slice hyperholomorphic functions}

In this section, we investigate whether the space obtained by applying the Fueter operator $D$ to the set of slice hyperholomorphic functions, namely $\mathcal{AH}_D^L(\Omega)$ (see \eqref{AAH}), coincides with the set of axially harmonic functions (see the definition below).

\begin{definition}
	Let $\Omega \subseteq \mathbb{H}$ be an open set. A function $f: \Omega \to \mathbb{H}$ is said to be axially harmonic on $\Omega$ if it is of left (or right) axial type (see Definition~\ref{axial1}) and satisfies
$$
	\Delta_4 f(q)=0.
$$
	We denote this class of functions by $\mathcal{AH}_L(\Omega)$ (resp. $\mathcal{AH}_R(\Omega)$).
\end{definition}

The problem we address in this section is to determine whether the following operator
\begin{equation}
	D: \mathcal{SH}_L(\Omega) \longrightarrow \mathcal{AH}_L(\Omega)
\end{equation}
is surjective. However, the following result shows that the operator $D$, when applied to the class of slice hyperholomorphic functions, does not yield the entire class of axially harmonic functions.

\begin{lemma}
	Let $\Omega\subseteq \mathbb H$ be an axially symmetric domain such that $\Omega\cap\mathbb R\neq\emptyset$. Then we have that

\begin{equation}
	\label{pinc}
	\mathcal{AH}^L_{D}(\Omega) \subsetneq \mathcal{AH}_L(\Omega).
\end{equation}	
\end{lemma}

\begin{proof}
As in Lemma~\ref{inclusion}, the strict inclusion is demonstrated by the function $f(q)=\underline{q}$.
\end{proof}

As in the case of the operator $\overline{D}$ in Section~6, our aim is to identify a subspace of axially harmonic functions whose direct sum with $\mathcal{AH}_D^L(\Omega)$ coincides with the full space of axially harmonic functions.

\begin{theorem}\label{t2_2}
	Let $\Omega\subseteq \mathbb H$ be an axially symmetric domain such that $\tilde\Omega=\Omega\cap\mathbb R\neq\emptyset$. Let $\Omega_1$ be a neighborhood of $\tilde \Omega$ defined as in \eqref{Stefano1}. Then every function $f\in\mathcal {AH}_L(\Omega)$ admits a unique decomposition of the form
\begin{equation}
\label{decoH}
f(q)=Dg_1(q)+g_2(q),\quad q\in\Omega_1
\end{equation}
	where $g_1\in\mathcal{SH}_L(\Omega_1)$ (determined up to an additive constant) and $g_2\in\mathcal{AH}_L(\Omega_1)$ satisfies
	$$ {g_2}_{|_{\unq=0}}=0 .$$
\end{theorem}

\begin{proof}
	The proof follows by a straightforward adaptation of the argument used in Theorem \ref{t2}.
\end{proof}

\begin{definition}
	\label{DDprimitive}
	Let $\Omega_1$ as in \eqref{Stefano1}. The function $g_1\in \mathcal{SH}_L(\Omega_1)$ is called the weak $D$-primitive of $f\in \mathcal{AH}_L(\Omega_1)$  if $( D{g_1})|_{\underline q=0}=f |_{\underline q=0}$.
\end{definition}

So by the above result, in special neighbourhood of the real axis, we can write
$$ \mathcal{AH}_L(\Omega_1)=\mathcal{AH}_D^L(\Omega_1)\oplus \mathcal {AH}_0^L(\Omega_1) ,$$
where
$$ \mathcal {AH}_0^L(\Omega_1):=\{g_2 \in  \mathcal{AH}_L(\Omega_1) \, : \, g_2 |_{\underline{q}=0}=0\}.$$

In \cite[Thm. 3.1]{DG}, the authors introduce an isomorphism that yields the full class of axially harmonic functions, namely the GCK-extension for axially harmonic functions; see the result below.

\begin{theorem}
Let $\Omega$ be an intrinsic complex domain. We set $\tilde{\Omega}=\Omega \cap \mathbb{R}$. We consider $A_0$, $A_1 \in \mathcal{A}(\tilde{\Omega}) \otimes \mathbb{H}$. Then there exists a unique sequence of functions $ \{A_j\}_{j \in \mathbb{N}_0}$ such that the series
$$ f(q)= \sum_{j=0}^{\infty} \underline{q}^j A_j(q_0),$$
converges in an axially symmetric $4$-dimensional neighbourhood $\Omega_1 \subset \mathbb{H}$ of $\tilde{\Omega}$ such that $f(q)$ is harmonic (i.e. $\Delta_4 f(q)=0$) in $\Omega_1$. Moreover
\begin{equation}
\label{HGCK}
f(q)= \frac{\sqrt{\pi}}{2} \left[\sum_{j=0}^\infty  \frac{(-1)^j | \underline{q}|^{2j} \partial_{q_0}^{2j}}{2^{2j}j! \Gamma \left(\frac{3}{2}+j\right)}A_0(q_0)+ \frac{3 \underline{q}}{2} \sum_{j=0}^\infty  \frac{(-1)^j | \underline{q}|^{2j} \partial_{q_0}^{2j}A_1(q_0)}{2^{2j}j! \Gamma\left(j+\frac{5}{2}\right)}\right],
\end{equation}
and the initial functions can be recovered by
$$ A_0(q_0)= \lim_{|\underline{q}| \to 0} f(q), \qquad A_1(q_0)= -\frac{1}{3} \partial_{\underline{q}}[f(q)]|_{\underline{q}=0}.$$
The function in \eqref{HGCK} is the harmonic GCK-extension of the couple $(A_0,A_1)$, and is denoted by $HGCK[A_0, A_1](q)$. This extension operator defines an isomorphism between right modules:

$$ HGCK \, : \, \left(\mathcal{A}(\tilde{D}) \otimes \mathbb{H}\right)^2 \to \mathcal{AH}_L(\Omega_1),$$
where the set $\Omega_1$ defined as in \eqref{Stefano1}.
\end{theorem}

Now, by means of the GCK-extension for axially harmonic functions we provide a basis for the space $ \mathcal{AH}_0^L(\Omega_1)$.

\begin{lemma}
\label{appH1}
Let $\Omega_1$ be defined as in Theorem \ref{t2_2}. The polynomials given by
\begin{equation}
\label{harmonic1}
H_k(q)= \frac{1}{(k+1)(k+2)(k+3)} \sum_{j=0}^{k+1} (k-2j+1) q^{k+1-j} \bar{q}^j,
\end{equation}
generate a basis for the space $ \mathcal{AH}_0^L(\Omega_1)$.
\end{lemma}
\begin{proof}
In \cite[Prop. 6.13]{DG} the authors have computed $HGCK[0, q_0^k](q)=H_k(q)$, with $k \in \mathbb{N}$. This give rise to the basis defined in \eqref{harmonic1}.
\end{proof}

In order to give an example of an axially harmonic function written in terms of the basis of $ \mathcal{AH}_D^L(\Omega_1)$ and \eqref{harmonic1} we need to recall the following result, see \cite{B, CDS}, and observe that the polynomials in \eqref{harmonic1} can be related to the Clifford-Appell polynomilas, see \eqref{CliffApp}.

\begin{lemma}
\label{appH}
Let $k \in \mathbb{N}$ then we have that
\begin{equation}
\label{harmpoly}
D q^k=-2k \mathcal{H}_{k-1}(q), \qquad \mathcal{H}_k(q)=\frac{1}{k+1} \sum_{j=0}^{k} q^{k-j} \bar{q}^j.
\end{equation}
\end{lemma}

\begin{lemma}
Let $k \in \mathbb{N}$ then we have
\begin{equation}
\label{relharm1}
H_k(q)=\frac{1}{k+1} \left[ \mathcal{Q}_{k+1}(q)-\mathcal{H}_{k+1}(q)\right].
\end{equation}
\end{lemma}
\begin{proof}
By manipulating \eqref{harmonic1} and using \eqref{CliffApp} and the polynomials in \eqref{harmpoly}, we obtain
\begin{eqnarray*}
H_k(q)
&=& \frac{1}{(k+1)(k+2)(k+3)} \sum_{j=0}^{k+1} [2(k+2-j)-(k+3)]q^{k+1-j} \bar{q}^j\\
&=& \frac{1}{k+1} \left[\frac{2}{(k+2)(k+3)} \sum_{j=0}^{k+1} (k+2-j) q^{k+1-j} \bar{q}^j-\frac{1}{k+2} \sum_{j=0}^{k+1} q^{k+1-j} \bar{q}^j \right]\\
&=&\frac{1}{k+1} \left[ \mathcal{Q}_{k+1}(q)-\mathcal{H}_{k+1}(q)\right].
\end{eqnarray*}
\end{proof}
\begin{remark}
By \eqref{relharm1}, it is clear that the restriction of the polynomials to the real line is zero. Indeed, by \eqref{rest} we have $\mathcal{Q}_{k+1}(q_0)=q_0^{k+1}$, and by \eqref{relharm1} we obtain $\mathcal{H}_{k+1}(q_0)=q_0^{k+1}$.
\end{remark}

By Theorem~\ref{t2_2}, Lemma~\ref{appH}, and Lemma~\ref{appH1}, we obtain the following result.

\begin{theorem}
	Let $\Omega$ and $\Omega_1$ be as in Theorem~\ref{t2}. Then any function $f \in \mathcal{AH}_L(\Omega)$ can be written as
	$$
	f(q)= \sum_{m=0}^\infty \mathcal{H}_m(q)a_m+\sum_{n=0}^{\infty} H_n(q)b_n, \qquad  \{b_n\}_{n \geq 0}, \, \{a_m\}_{m \geq 0} \subseteq \mathbb{H}.
	$$
\end{theorem}

In the following example, we provide an example illustrating the previous result.

\begin{example}
	Let us consider
$$
	f(q)=  \mathcal{Q}_2(q)+q \mathcal{Q}_1(q).
$$
	It is clear that this function is of axial type. Moreover, it is harmonic since, by the Fueter regularity of the Clifford--Appell polynomials and \cite[Thm. 2.82]{CDS2025} we obtain
$$
	\Delta_4 f(q)= \Delta_4( q \mathcal{Q}_1(q))
	= q \Delta_4 (\mathcal{Q}_1(q)) - 2 D (\mathcal{Q}_1(q)) = 0.
$$
	Now, we aim to express the function in terms of the polynomials $\mathcal{H}_k(q)$ (see \eqref{harmpoly}) and $H_k(q)$ (see \eqref{harmonic1}). By \eqref{relharm1} we have
	\begin{equation}
	\label{exx}
	f(q)= 2 H_1(q)+ \mathcal{H}_2(q)+q \mathcal{Q}_1(q).
	\end{equation}
Now, by \eqref{CliffApp} and \eqref{harmpoly} we get
\begin{eqnarray}
\nonumber
q \mathcal{Q}_1(q)+ \mathcal{H}_2(q)&=& \frac{1}{3} \sum_{j=0}^1 (2-j) q^{2-j} \bar{q}^j+ \frac{1}{3} \sum_{j=0}^2 q^{2-j} \bar{q}^j\\
\nonumber
&=& \frac{1}{3} \sum_{j=0}^2 (3-j) q^{2-j} \bar{q}^j\\
\nonumber
&=& \frac{1}{3} \left( \sum_{j=0}^{2} (1-j) q^{2-j} \bar{q}^{j}+2 \sum_{j=0}^{2} q^{2-j} \bar{q}^j\right)\\
\label{exx2}
&=& 4 H_1(q)+2\mathcal{H}_2(q)
\end{eqnarray}	
By plugging \eqref{exx2} into \eqref{exx} we get
$$ f(q)= 2 \mathcal{H}_2(q)+6 H_1(q).$$
\end{example}
Our aim is to provide an explicit integral representation of the weak $D$-primitive; see Definition~\ref{DDprimitive}. To this end, we first introduce a family of kernels (see Definition \ref{d1_2}) and study their restriction to the real axis. The following theorem provides the motivation for their construction.

\begin{theorem} \label{rep_har}
	Let $f$ be a harmonic function defined in a neighborhood of the ball $B(0,1)\subset \mathbb{H}$. Then
	\begin{equation}
		\label{hp}
		f(q)=\frac{1}{2 \pi^2} \int_{\partial B(0,1)} \frac{1-|q|^2}{|q-\xi|^4} \, f(\xi)\, d\Sigma,
	\end{equation}
	where $d\Sigma$ denotes the (non-oriented) surface measure on $\partial B(0,1)$.
\end{theorem}
\begin{proof}
	It follows by using the classical formula components by components on the quaternionic-valued function $f$.
\end{proof}

\begin{definition} \label{d1_2}
	Let $q = q_0 + \unq \in \mathbb{H}$ and $\underline{\omega} \in \mathbb{S}$. We define the following kernels:
	\begin{align*}
		\mathcal H^+_1 (q) &=\frac{1}{2 \pi^2} \int_{\mathbb S} \frac{dS_{\underline{\omega}}}{|q-\underline{\omega}|^4},
		\quad
		\mathcal H^+_2 (q, \varphi) = \frac{1}{2 \pi^2} \int_{\mathbb S} \frac{|q+c(\varphi)|^2 \, dS_{\underline{\omega}}}{|q-\underline{\omega}|^4}, \\
		\mathcal H^-_1 (q) &= \frac{1}{2 \pi^2} \int_{\mathbb S} \frac{\underline{\omega} \, dS_{\underline{\omega}}}{|q-\underline{\omega}|^4},
		\quad
		\mathcal H^-_2 (q, \varphi) = \frac{1}{2 \pi^2} \int_{\mathbb S} \frac{|q+c(\varphi)|^2 \, \underline{\omega} \, dS_{\underline{\omega}}}{|q-\underline{\omega}|^4},
	\end{align*}
	where $dS_{\underline{\omega}}$ denotes the surface measure on $\mathbb{S}$, $\varphi$ is the polar angle of the sphere $\mathbb{S}$ and $c(\varphi)=\cot \varphi$, with $\varphi \in (0, \pi)$.
\end{definition}
For the proof of the next result, we need the following.
\begin{theorem}[Funk-Hecke Theorem, see \cite{HOCH}]\label{F_H} Let $\xi$ and $\eta$ be two unit vectors in $\mathbb R^3$. Let $\psi$ be a real-valued function whose domain contains $[-1,1]$ and let $S_m(\xi)$ be spherical harmonics, of degree $m$. Then we have
	$$ \int_{\mathbb S}\psi(\langle \eta,\xi \rangle) S_m(\eta) dS_\eta=2 \pi S_m(\xi)\int_{-1}^1 \psi(t) P_m(t)dt, $$
	where $dS_\eta$ is the scalar element of surface area on $\mathbb S$, $\langle\xi,\eta \rangle$ denotes the scalar product of $\xi$, $\eta$ and $P_m(t)$ is defined by the formula:
	$$ P_m(t)=\left( -\frac 12 \right)^m \frac{1}{\Gamma(m+1)} D^m(1-t^2)^{m}. $$
\end{theorem}

\begin{proposition}\label{p1_2} Let $q\in\mathbb H$ be a quaternion such that $q=x+Iy$. The restrictions of the kernels introduced in Definition \ref{d1_2} to the real axis are given by
	\begin{align*}
		\lim_{y\to 0}\mathcal H^{+}_1 (x+Iy)=   \frac{2}{\pi(x^2+1)^2} & \quad\quad
		\lim_{y\to 0}\mathcal H^{-}_1 (x+Iy)= 0\\
		\lim_{y\to 0}\mathcal H^{+}_2 (x+Iy,\varphi)= \frac{2 (x+c(\varphi))^2}{\pi(x^2+1)^2} &\quad\quad
		\lim_{y\to 0}\mathcal H^{-}_2 (x+Iy,\varphi)=0.
	\end{align*}
\end{proposition}
\begin{proof}
We start focusing on $\mathcal H^{+}_{1}(x+Iy)$. By the Funk-Hecke Theorem, see Theorem \ref{F_H}, setting $m=0$, we have
	\begin{align*}
		\mathcal H^+_1(x+Iy)&= \frac{1}{2 \pi^2}\int_{\mathbb S} \frac{dS_{\underline{\omega}}}{|x+Iy-\underline{\omega}|^4}\\
		&=\frac{1}{2 \pi^2}\int_{\mathbb S} \frac{dS_{\underline{\omega}}}{\left(x^2+y^2+1-2y\langle I,\underline{\omega} \rangle\right)^2}\\
		&= \frac{1}{\pi} \int_{-1}^1 \frac{1}{\left( x^2+y^2+1-2yt\right)^2} dt,
	\end{align*}
where we have set $t=\langle I,\underline{\omega} \rangle$. Thus, we have
	$$\lim_{y\to 0}\mathcal H^{+}_1 (x+Iy)=  \frac{1}{\pi} \frac{\int_{-1}^1 1\, dt}{(x^2+1)^2}=  \frac{2}{\pi(x^2+1)^2}. $$
We focus on $\mathcal{H}^-_1(x+Iy)$. By Theorem \ref{F_H}, setting $m=1$ we get
	\begin{align*}
		\mathcal H^-_1(x+Iy)&= \frac{1}{2 \pi^2}\int_{\mathbb S} \frac{\underline{\omega} dS_{\underline{\omega}}}{|x+Iy-\underline{\omega}|^4}\\
		&=\frac{1}{2 \pi^2}\int_{\mathbb S} \frac{\underline{\omega} dS_{\underline{\omega}}}{\left(x^2+y^2+1-2y\langle I,\underline{\omega} \rangle\right)^2}\\
		&=\frac{I}{\pi} \int_{-1}^1 \frac{t}{\left( x^2+y^2+1-2yt\right)^2} dt,
	\end{align*}
where we have set $t=\langle I,\underline{\omega} \rangle$. Thus, we have
	$$ \lim_{y\to 0}\mathcal H^{-}_1 (x+Iy)=   \frac{I }{2 \pi} \frac{\int_{-1}^1  t\, dt}{(x^2+1)^2}=0. $$
	Now, we focus our attention to $\mathcal H^{+}_{2}(x+Iy, \varphi)$. By Theorem \ref{F_H}, setting $m=0$ we have
	\begin{align*}
		\mathcal H^+_2(x+Iy,\varphi)&= \frac{1}{2 \pi^2}\int_{\mathbb S} \frac{|x+Iy+c(\varphi)|^2 dS_{\underline{\omega}}}{|x+Iy-\underline{\omega}|^4}\\
		&=\frac{1}{2 \pi^2} \int_{\mathbb S} \frac{ |x+Iy+c(\varphi)|^2 dS_{\underline{\omega}}}{\left(x^2+y^2+1-2y\langle I,\underline{\omega} \rangle\right)^2}\\
		&= \frac{1}{\pi} \int_{-1}^1 \frac{|x+Iy+c(\varphi)|^2}{\left( x^2+y^2+1-2yt\right)^2} dt,
	\end{align*}
where we have set $t=\langle I,\underline{\omega} \rangle$. Thus, we have
	$$ \lim_{y\to 0}\mathcal H^{+}_2 (x+Iy,\varphi)= \frac{ (x+c(\varphi))^2}{\pi} \frac{\int_{-1}^1 1\, dt}{(x^2+1)^2}=  \frac{2 (x+c(\varphi))^2}{\pi(x^2+1)^2}. $$
Finally, we turn our attention to $\mathcal{H}^-_2(x+Iy,\varphi)$. By Theorem~\ref{F_H}, and choosing $m=0$ , we obtain
	\begin{align*}
		\mathcal H^-_2(x+Iy,\varphi)&= \frac{1}{2 \pi^2}\int_{\mathbb S} \frac{ |x+Iy+c(\varphi)|^2 \underline{\omega} dS_{\underline{\omega}}}{|x+Iy-\underline{\omega}|^4}\\
		&=\frac{1}{2 \pi^2}\int_{\mathbb S} \frac{ |x+Iy+c(\varphi)|^2 \underline{\omega} dS_{\underline{\omega}}}{\left(x^2+y^2+1-2y\langle I,\underline{\omega} \rangle\right)^2}\\
		&= \frac{I |x+Iy+c(\varphi)|^2}{\pi} \int_{-1}^1 \frac{t}{\left( x^2+y^2+1-2yt\right)^2} dt,
	\end{align*}
where we have set $t=\langle I,\underline{\omega} \rangle$.
	Thus, we have
	$$ \lim_{y\to 0}\mathcal H^{-}_2 (x+Iy,\varphi)= \frac{I(x+c(\varphi))^2 }{2 \pi} \frac{\int_{-1}^1  t\, dt}{(x^2+1)^2}=0. $$
\end{proof}
\begin{proposition}\label{restriction_2}
	Let $q\in\mathbb H$ be a a quaternion such that $q=q_0+Iq_1$. The weak $ D$-primitives of the kernels $\mathcal H^{+}_{1,2}(q)$ are given by
$$
		h_{1}^+(q)=-\frac{1}{2 \pi}\left( \arctan(q)+ \frac{q}{1+q^2} \right)
$$
	and
	\begin{align*}
		h_{2}^+(q)=-\frac{1}{2 \pi}\left((1+c^2(\varphi)) \arctan(q)+ (c^2(\varphi)-1)\frac{q}{1+q^2} -\frac{2c(\varphi)}{q^2+1}\right),
	\end{align*}
where $c(\varphi)=\cot \varphi$, with $\varphi \in (0, \pi)$.
\end{proposition}
\begin{proof}
	
	We seek the slice hyperholomorphic function $h_1^+(q)$ such that $\left(D h_1^+(q)\right)_{|_{\unq=0}}=\mathcal H^+_1(q_0)$. Since in the ball $B(0,1)$ we have the expansion in series \eqref{exp}, by \cite[Lemma 2.81]{CDS2025} we observe that
	$$ \left(D h_1^+(q)\right)_{|_{\unq=0}}=-2\partial_{q_0}h_1^+ (q_0). $$
	Thus, by using the formula for $\mathcal H^+_1(q_0)$ obtained in Proposition \ref{p1_2}, we can conclude that
	$$ h_1^+(q_0)=- \frac{1}{\pi} \int_{0}^{q_0} \frac{1}{(s_0^2+1)^2}\, ds_0 =-\frac{1}{2\pi}\left( \arctan(q_0)+ \frac{q_0}{1+q_0^2} \right). $$
	The slice hypeholomoprhic function which extends $h_1^+(q_0)$ is obtained just replacing the variable $q_0$ with the quaternionic variable $q$, see \cite{CSSisrael2}.
	
Now we can determine the slice hyperholomorphic function $h_2^+(q)$ such that $\left(Dh_2^+(q)\right)_{\mid \underline{q}=0}=\mathcal{H}_2^+(q_0)$.
Since in the ball $B(0,1)$ we have the series expansion \eqref{exp}, and by \cite[Lemma 2.81]{CDS2025}, we obtain
	$$ \left( D(h_2^+(q)) \right)_{| \unq=0}=-2\partial_{q_0} h_2^+ (q_0). $$
	Thus, using the formula for $\mathcal H^+_2(q_0)$ obtained in Proposition \ref{p1_2}, we can conclude that
	\begin{align*}
		h_2^+(q_0) & =- \frac{1}{\pi} \int_{0}^{q_0} \frac{(s_0+c(\varphi))^2}{(s_0^2+1)^2}\, ds_0\\
		&  =-\frac{1}{2\pi}\left((1+c^2(\varphi)) \arctan(q_0)+ (c^2(\varphi)-1)\frac{q_0}{1+q_0^2} -\frac{2c(\varphi)}{q_0^2+1}\right).
	\end{align*}
	The slice hyperholomorphic functions which extends $h_2^+(q_0)$ is obtained just replacing the variable $q_0$ with the quaternionic variable $q$, see \cite{CSSisrael2}.
\end{proof}

\begin{remark}
\label{prim}
We denote by $h_1^{-}(q)$ and $h_2^{-}(q)$ the weak $D$-primitives of $\mathcal{H}_1^{-}(q)$ and $\mathcal{H}_2^{-}(q)$, respectively. By Proposition~\ref{p1_2}, it follows that $h_1^{-}(q)=h_2^{-}(q)=0$.
\end{remark}

For a general domain, no explicit formula analogous to the Cauchy formula is available for the kernel associated with harmonic functions, as it depends on the Green function of the domain. Consequently, in this setting it is not possible, in general, to derive an explicit expression for the weak harmonic primitive of an axially harmonic function, as in the case of axially polyanalytic functions of order $2$ treated in the previous section.

For this reason, we restrict our attention to the unit ball centred at the origin. In this case, an explicit integral representation for harmonic functions is available, and has been stated in Theorem \ref{rep_har}.

Now, we have all the tools to get an integral representation for the weak $D$-primitive.

\begin{theorem}\label{harmonic_primitive}
	Let $f \in \mathcal{AH}_L(\overline{B(0,1)})$. Then the weak $D$-primitive of $f$, according to the notation of Theorem \ref{t2_2}, is given by
	\begin{align*}
		g_1(q):=\int_0^\pi \left( h_1^+\!\left(\frac{q-\cos(\varphi)}{\sin(\varphi)} \right)\frac{1}{\sin(\varphi)}- h_2^+\!\left(\frac{q-\cos(\varphi)}{\sin(\varphi)} \right) \sin(\varphi)  \right) A(\cos(\varphi),\sin(\varphi))\, d\varphi,
	\end{align*}
	where $h_1^+$ and $h_2^+$ are defined in Proposition~\ref{restriction_2}, and the function $A(\cos(\varphi),\sin(\varphi))$ is the component of the axial function $f$, see Definition \ref{sh}
\end{theorem}

\begin{proof}
We start by parametrizing the boundary $\partial B(0,1)$ as follows:
$$
	(\varphi,\underline{\omega})\mapsto \cos(\varphi)+\underline{\omega}\sin(\varphi), \qquad (\varphi,\underline{\omega})\in [0,\pi]\times \mathbb S.
$$
	With this parametrization, the surface element becomes $d\Sigma=\sin^2(\varphi)\, d\varphi\, dS_{\underline{\omega}}$. Hence formula \eqref{hp} can be written as
	\begin{align*}
		f(q) &=\frac{1-|q|^2}{2 \pi^2}\int_0^{\pi} \int_{\mathbb S}
		\frac{ A(\cos\varphi,\sin\varphi)+\underline{\omega} B(\cos\varphi,\sin\varphi)}
		{\bigl|q-\cos\varphi-\underline{\omega} \sin\varphi\bigr|^4}
		\, dS_{\underline{\omega}}\, \sin^2(\varphi)\, d\varphi.
	\end{align*}
	Factoring out $\sin(\varphi)$ in the denominator and introducing the change of variable
	\[
	q'=\frac{q-\cos(\varphi)}{\sin(\varphi)},
	\]
	we obtain
	\begin{align}
	\label{intrem}
		f(q)
		&=\frac{1}{2 \pi^2}\int_0^{\pi} \int_{\mathbb S}
		\frac{ A(\cos\varphi,\sin\varphi)+\underline{\omega} B(\cos\varphi,\sin\varphi)}
		{|q'-\underline{\omega}|^4}
		\, dS_{\underline{\omega}}\, \frac{d\varphi}{\sin^2(\varphi)} \\
		\nonumber
		&\quad -\frac{1}{2 \pi^2 }\int_0^{\pi} \int_{\mathbb S}
		\frac{|q' + c(\varphi)|^2\bigl( A(\cos\varphi,\sin\varphi)+\underline{\omega} B(\cos\varphi,\sin\varphi)\bigr)}
		{|q'-\underline{\omega}|^4}
		\, dS_{\underline{\omega}}\, d\varphi,
	\end{align}
	where $c(\varphi)=\cot \varphi$. Now by Definition~\ref{d1_2}, we have
	\begin{align*}
		f(q)= &\int_0^{\pi}  \left( \mathcal H^+_1 (q')  A(\cos\varphi,\sin\varphi) + \mathcal H^-_1 (q',\varphi) B(\cos\varphi,\sin\varphi)  \right) \frac{d\varphi}{\sin^2(\varphi)}\\
		&\quad -\int_0^{\pi}  \left( \mathcal H^+_2 (q',\varphi)  A(\cos\varphi,\sin\varphi) + \mathcal H^-_2 (q',\varphi) B(\cos\varphi,\sin\varphi)  \right) d\varphi.
	\end{align*}
Now, by Proposition~\ref{restriction_2}, Remark \ref{prim}, Definition \ref{DDprimitive} and the fact that $D_{q'}=\sin(\varphi)\, D_q$ we have	

\begin{eqnarray*}
Dg_1(q)&=&D_{q'}\left[\int_0^{\pi}   \left( h^+_1 (q')   \frac{1}{\sin^2(\varphi)} -  h^+_2 (q')   \right) A(\cos\varphi,\sin\varphi) d\varphi \right]\\
&=&D_{q} \left[\int_0^{\pi} \left(h_1^{+}(q') \frac{1}{ \sin \varphi}- h^+_2 (q')   \sin \varphi\right)A(\cos\varphi,\sin\varphi)d\varphi \right].
\end{eqnarray*}
The function $g_1(q)$, defined by
$$
g_1(q):=\int_0^\pi \left( h_1^+\!\left(\frac{q-\cos(\varphi)}{\sin(\varphi)} \right)\frac{1}{\sin(\varphi)}- h_2^+\!\left(\frac{q-\cos(\varphi)}{\sin(\varphi)}\right) \sin(\varphi) \right) A(\cos(\varphi),\sin(\varphi))\, d\varphi,
$$
is slice hyperholomorphic, since the functions $h_1^{+}$ and $h_2^{+}$ are slice hyperholomorphic by construction. Hence, $g_1$ is the weak $D$-primitive of $f$.
\end{proof}

\begin{remark}
We note that the integrand in \eqref{intrem} involves division by $\sin(\varphi)$; therefore, a more rigorous approach would be to perform the integration over the interval $[\varepsilon, \pi - \varepsilon]$ and then pass to the limit as $\varepsilon \to 0$. For the sake of simplicity, however, we keep the integration over the full interval $[0,\pi]$.
\end{remark}

\section{On the map $\mathcal{C}_2$}

The class of axially polyanalytic functions of order $2$ can also be obtained by means of an operator that suitably exploits the polyanalytic decomposition; see Theorem~\ref{polydeco}. In order to give a precise formulation, we first recall some specific notions.

\medskip

In \cite{ADS, ADS2019, ACDS2}, the authors extend the notion of polyanalytic functions to the slice setting. More precisely, they introduce the following definition.

\begin{definition}[Slice polyanalytic functions]
	\label{slicepoly}
	Let $n \in \mathbb{N}$ and $\Omega \subseteq \mathbb{H}$ be an axially symmetric open set. We say that a slice function of the form $f(q)=\alpha(u,v)+I\beta(u,v)$ in $\mathcal{C}^n(\Omega)$ is a left slice polyanalytic function of order $n$ if the functions $\alpha$ and $\beta$ satisfy the even-odd conditions \eqref{eo} and the poly-Cauchy-Riemann equation
	$$
	\left(\frac{\partial}{\partial u}+I \frac{\partial}{\partial v}\right)^n (\alpha(u,v)+I \beta(u,v))=0, \qquad \forall I \in \mathbb{S}.
	$$
	The definition of right slice polyanalytic functions of order $n$ can be easily adapted. The set of left (resp. right) slice polyanalytic functions of order $n$ is denoted by $\mathcal{SP}_n^L(\Omega)$ (resp. $\mathcal{SP}_n^R(\Omega)$).
\end{definition}

\begin{remark}
	If we consider $n=1$ in Definition~\ref{slicepoly}, we recover the definition of left slice hyperholomorphic functions; see Definition~\ref{sh}.
\end{remark}

\begin{theorem}
	\label{bapoly}
	A function $f$ is slice polyanalytic of order $n$ if and only if there exists left slice hyperholomorphic functions $f_0$,..., $f_{n-1}$ such that the function $f$ can be decomposed as
	$$ f(q)=\sum_{k=0}^{n-1} \overline{q}^k f_k(q).$$
\end{theorem}

\begin{remark}
	In view of the above decomposition, an example of a slice polyanalytic function of order $n+1$ is given by the Clifford--Appell polynomials; see \eqref{CliffApp}.
\end{remark}

\begin{remark}
	\label{Global}
	A connection between left slice polyanalytic functions of order $n+1$ and left slice hyperholomorphic functions is given by the so-called left global operator; see \cite{ADS1}. This operator is defined as
	$$
	V := \frac{\partial}{\partial q_0} + \frac{\underline{q}}{|\underline{q}|^2}\left( \sum_{\ell=1}^3 q_\ell \frac{\partial}{\partial q_\ell}\right), \qquad q \in \mathbb{H} \setminus \mathbb{R}.
	$$
	More precisely, we have
	$$
	\begin{CD}
		\mathcal{SP}_{n+1}^L(\Omega) @> V^n >> \mathcal{SH}_L(\Omega)
	\end{CD}
	$$
	The global operator was introduced in \cite{CGS} and is written in a different form in \cite{GP}.
\end{remark}

In \cite{ADS1}, the authors present another way to obtain axially polyanalytic functions of order $2$, which is given by the following operator:

\begin{eqnarray}
	\label{polyope}
	\mathcal{C}_2&:& \mathcal{SP}_2^L(\Omega) \to \mathcal{APA}_2^L(\Omega)\\
	\nonumber
	&& \sum_{k=0}^{1} \bar{q}^k f_k(q) \mapsto \sum_{k=0}^{1} q_0^k \Delta_4 f_k(q),
\end{eqnarray}
where $f_0$ and $f_1$ are left slice hyperholomorphic functions.

\begin{remark}
In \cite{DK3}, the above operator is connected with the theory of the Bargmann transform.
\end{remark}

The main difference with the map in \eqref{mapfine} is that, in that case, the domain is the space of slice hyperholomorphic functions, whereas in the present case the map $\mathcal{C}_2$ has as its domain the class of slice polyanalytic functions of order $2$. Moreover, the main point is that, by using the map $\mathcal{C}_2$, one can obtain the full class of axially polyanalytic functions of order $2$ starting from slice polyanalytic functions of order $2$, as is proved in the following results.

\begin{theorem}
	The operator $\mathcal{C}_2$ is surjective.
\end{theorem}
\begin{proof}
	We have to prove that given a function $g \in \mathcal{APA}_2^L(\Omega)$ there exists a function $f \in \mathcal{SP}_2^L(\Omega)$ such that
	$$ g(q)=\mathcal{C}_2(f)(q).$$
	By Theorem \ref{polydeco} we can write the function $g$ as
	$$ g(q)= \sum_{k=0}^{1} q_0^k g_k(q),$$
	where $g_0$ and $g_1$ are axially Fueter regular functions. By \cite{CSSinve} there exist $f_0$, $f_1 \in \mathcal{SH}_L(\Omega)$ such that
	$$g_0(q)=\Delta_4 f_0(q), \qquad g_1(q)= \Delta_4 f_1(q).$$
	Therefore we have that
	$$ g(q)= \sum_{k=0}^{1} q_0^k g_k(q)=\sum_{k=0}^{1} q_0^k \Delta_4(f_k(q))=\mathcal{C}_2(f)(q).$$
	This proves the result where $f(q):=f_0(q)+\bar q f_1(q)$.
\end{proof}

\begin{theorem}
	The operator $ \mathcal{C}_2$ is bijective up to the kernel of $\Delta_4$.
\end{theorem}
\begin{proof}
	We have proved in the previous result that the map $\mathcal{C}_2$ is surjective; thus, it remains to prove that $\mathcal{C}_2$ is injective up to the kernel of $\Delta_4$. Let $f, g \in \mathcal{SP}_2^L(\Omega)$ be such that $\mathcal{C}_2(f)(q) = \mathcal{C}_2(g)(q)$. We must show that $f(q) = g(q)$ up to the kernel of $\Delta_4$. By the definition of the map $\mathcal{C}_2$, we have
	$$
	\Delta_4 f_0 = \Delta_4 g_0, \qquad \Delta f_1 = \Delta g_1,
	$$
	where $f_0, f_1, g_0, g_1 \in \mathcal{SH}_L(\Omega)$. The above equalities imply that
	$$
	f_0 = g_0 + \mathrm{Ker}(\Delta_4), \quad f_1 = g_1 + \mathrm{Ker}(\Delta_4).
	$$
	Hence, by Theorem~\ref{bapoly}, we conclude that $f(q) = g(q)$ up to $\ker(\Delta_4)+q_0\ker(\Delta_4)$.
\end{proof}

\begin{remark}
In \cite{DDG1}, it has been shown that the kernel of $\Delta_4$ is given by affine functions, i.e., functions of the form $aq + b$, with $a, b \in \mathbb{H}$.
\end{remark}

\begin{remark}
	By Remark~\ref{Global}, a connection between the map $\mathcal{C}_2$ and \eqref{mapfine} is given by the following diagram:
	\begin{center}
		\begin{tikzcd}[row sep=3em, column sep=6em]
			\mathcal{SP}_2(\Omega)
			\arrow[r, rightarrow, "\mathcal{C}_2"]
			\arrow[d, "V" swap]
			& \mathcal{APA}_2^L(\Omega)
			\arrow[d, hookleftarrow, ""] \\
			\mathcal{SH}_L(\Omega)
			\arrow[r, "\overline{D}"]
			& \mathcal{APA}_{2, \overline{D}}^L(\Omega)
		\end{tikzcd}
	\end{center}
	
\end{remark}

\section{Concluding remarks}

In this paper we have shown that, although the Fueter map $\Delta_4$ is surjective, its factorization in terms of $D$ and $\overline{D}$ is not. In order to obtain the full space of functions mapped by the operators $D$ and $\overline{D}$, when applied to the class of slice hyperholomorphic functions, we have considered a suitable subspace characterized by appropriate regularity conditions and by the requirement that the function restricted to the real axis vanishes. We have also seen, in one case, that another way to recover the full space of functions is to construct a suitable map connecting the corresponding space with its slice counterpart.

In a future work, we aim to consider polyanalytic functions of arbitrary order. To this end, we need study the factorization of the Fueter--Sce map, given by $\Delta_{n+1}^{\frac{n-1}{2}}$ with $n$ odd; see \cite{CDP25, Fivedim}. We also intend to investigate these problems in the context of the fine structures that appear in \cite{CDP2026}.

\section*{Appendix}

We recall the definition of the double factorial as follows:
$$ j!! =\begin{cases}
	1, \qquad j=0,1\\
	j \cdot (j-2)!!, \qquad j \geq 2.
\end{cases}
$$

\begin{proposition}
	\label{recc}
Let $j \geq 0$, let $\{A_j\}_{j \in \mathbb{N}_0}$ be a sequence of functions, and let the constants $c(j+1)$ and $k(j+2)$ be defined in \eqref{f1} and \eqref{f2}, respectively. Then, from the following recurrence relation
	\begin{equation}
		\label{rel}
		A_{j+2}(q_0)=-2 \frac{c(j+1)}{k(j+2)} \partial_{q_0} A_{j+1}(q_0)-\frac{\partial_{q_0}^2 A_j(q_0)}{k(j+2)},
	\end{equation}
	we get
	\begin{equation}
		\label{f3}
		A_{j+2}(q_0)=
		\begin{cases}
			\vspace*{3mm}
			\frac{3 \partial_{q_0}^{j+1} A_1(q_0)}{j!! (j+3)!!}-\frac{(j+1)\partial_{q_0}^{j+2}A_0(q_0)}{(j+2)!!(j+3)!!}, \qquad j  \, \, \hbox{is even}\\
			\frac{3 (j+2) \partial_{q_0}^{j+1} A_1(q_0)}{(j+1)!! (j+4)!!}- \frac{\partial_{q_0}^{j+2}A_0(q_0)}{(j-1)!! (j+4)!!}, \qquad j \, \, \hbox{is odd}
		\end{cases}
	\end{equation}
\end{proposition}
\begin{proof}
We distinguish two cases starting from \eqref{rel}. By the definition of the constants $c(j+1)$ and $k(j+1)$, see \eqref{f1} and \eqref{f2}, respectively, we obtain
	\begin{equation}
		\label{f4}
		A_{j+2}(q_0)
		=
		\begin{cases}
			\vspace{3mm}
			\frac{2}{j+2} \partial_{q_0} A_{j+1}(q_0)-\frac{\partial_{q_0}^2 A_j(q_0)}{(j+2)(j+3) }, \qquad j \, \, \hbox{is even}\\
			\frac{2}{j+4} \partial_{q_0} A_{j+1}(q_0)-\frac{\partial_{q_0}^2 A_j(q_0)}{(j+1)(j+4) }, \qquad j \, \, \hbox{is odd}
		\end{cases}
	\end{equation}
	We prove \eqref{f3} by induction on $j$. We begin by proving the cases $j=0$ and $j=1$. By \eqref{f4} we have
	\begin{equation}
		\label{rel2}
		A_2(q_0)= \partial_{q_0} A_1(q_0)-\frac{\partial_{q_0}^2 A_0(q_0)}{6}.
	\end{equation}
	This proves \eqref{f3} for the case $j=0$. By using \eqref{f4} and \eqref{rel2} we have
	$$ A_3(q_0)=\frac{2 \partial_{q_0} A_2(q_0)}{5}-\frac{\partial_{q_0}^2 A_1(q_0)}{10}=\frac{3\partial_{q_0}^2A_1(q_0)}{10}-\frac{\partial_{q_0}^3 A_0(q_0)}{15}.$$
	This proves \eqref{f3} in the case $j=1$. Now, we suppose that \eqref{f3} is true and we prove it for $j+1$. We divide this part of the proof in two cases: $j$ even and $j$ odd. We start from the case $j$ even. By \eqref{f4} and the inductive hypothesis we have
				\begingroup\allowdisplaybreaks
	\begin{eqnarray*}
		A_{j+3}(q_0)&=&A_{(j+1)+2}(q_0)\\
		&=& \frac{2 \partial_{q_0} A_{j+2}(q_0)}{(j+5)}-\frac{\partial_{q_0}^2 A_{j+1}(q_0)}{(j+2)(j+5)}\\
		&=& \frac{2 \partial_{q_0} A_{j+2}(q_0)}{(j+5)}-\frac{\partial_{q_0}^2 A_{(j-1)+2}(q_0)}{(j+2)(j+5)}\\
		&=& \frac{6 \partial_{q_0}^{j+2} A_1(q_0)}{(j+5) j!! (j+3)!!}- \frac{2(j+1) \partial_{q_0}^{j+3} A_0(q_0)}{(j+5) (j+2)!! (j+3)!!}\\
		&&-\frac{3(j+1)\partial_{q_0}^{j+2}A_1(q_0)}{(j+2) (j+5)j!! (j+3)!!}+\frac{\partial_{q_0}^{j+3}A_0(q_0)}{(j+2)(j+5) (j-2)!! (j+3)!!}\\
		&=& \frac{3[2(j+2)-(j+1)] \partial_{q_0}^{j+2}A_1(q_0)}{(j+2)!!(j+5)!!}- \frac{[2(j+1)-j] \partial_{q_0}^{j+3}A_0(q_0)}{(j+2)!! (j+5)!!}\\
		&=& \frac{3 (j+3) \partial_{q_0}^{j+2}A_1(q_0)}{(j+2)!! (j+5)!!}- \frac{\partial_{q_0}^{j+3}A_0(q_0)}{j!! (j+5)!!}.
	\end{eqnarray*}
	\endgroup
	This proves \eqref{f3} in the case where the index $j+1$ is odd. Now we prove \eqref{f3} in the case $j$ is odd. By \eqref{f4} and the inductive hypothesis we have
					\begingroup\allowdisplaybreaks
	\begin{eqnarray*}
		A_{j+3}(q_0)&=& A_{(j+1)+2}(q_0)\\
		&=& \frac{2 \partial_{q_0} A_{j+2}(q_0)}{j+3}-\frac{\partial_{q_0}^2 A_{j+1}}{(j+3)(j+4)}\\
		&=& \frac{2 \partial_{q_0} A_{j+2}(q_0)}{j+3}-\frac{\partial_{q_0}^2 A_{(j-1)+2}}{(j+3)(j+4)}\\
		&=& \frac{6(j+2) \partial_{q_0}^{j+2}A_1(q_0)}{(j+3)(j+1)!! (j+4)!!}-\frac{2 \partial_{q_0}^{j+3}A_0(q_0)}{(j+3)(j-1)!!(j+4)!!}\\
		&&- \frac{3 \partial_{q_0}^{j+2}A_1(q_0)}{(j+3)(j+4)(j-1)!!(j+2)!!}+ \frac{j \partial_{q_0}^{j+3} A_0(q_0)}{(j+3) (j+4) (j+2)!! (j+1)!!}\\
		&=& \frac{3[2(j+2)-(j+1)] \partial_{q_0}^{j+2} A_1(q_0)}{(j+3)!! (j+4)!!}- \frac{[2(j+1)-j] \partial_{q_0}^{j+3}A_0(q_0)}{(j+3)!!(j+4)!!}\\
		&=&\frac{3 \partial_{q_0}^{j+2} A_1(q_0)}{(j+1)!! (j+4)!!}- \frac{(j+2) \partial_{q_0}^{j+3}A_0(q_0)}{(j+3)!!(j+4)!!}.
	\end{eqnarray*}
	\endgroup
	This proves \eqref{f3} in the case where the index $j+1$ is even.

\end{proof}

\section*{Declarations and statements}

{\bf Conflict of Interest}. The authors declare that they have no competing interests regarding the publication of this paper.

{\bf Author contributions}. All authors contributed equally to the study, read and approved the final version of the submitted manuscript.

{\bf Availability of data}. There are no data associated with the research in this paper.


\end{document}